\documentclass[3p,11pt]{elsarticle}
\usepackage{amssymb}
\usepackage{amsmath}
\usepackage{amsthm}

\DeclareMathAlphabet{\mathcal}{OMS}{cmsy}{m}{n}

\makeatletter
\def\ps@pprintTitle{%
 \let\@oddhead\@empty
 \let\@evenhead\@empty
 \def\@oddfoot{\centerline{\thepage}}%
 \let\@evenfoot\@oddfoot}
\makeatother

\newcommand{\bbC}{\mathbb{C}}
\newcommand{\bbF}{\mathbb{F}}
\newcommand{\bbH}{\mathbb{H}}
\newcommand{\bbQ}{\mathbb{Q}}
\newcommand{\bbR}{\mathbb{R}}
\newcommand{\bbT}{\mathbb{T}}
\newcommand{\bbZ}{\mathbb{Z}}

\newcommand{\bfA}{\mathbf{A}}
\newcommand{\bfC}{\mathbf{C}}
\newcommand{\bff}{\mathbf{f}}
\newcommand{\bfE}{\mathbf{E}}
\newcommand{\bfF}{\mathbf{F}}
\newcommand{\bfg}{\mathbf{g}}
\newcommand{\bfG}{\mathbf{G}}
\newcommand{\bfI}{\mathbf{I}}
\newcommand{\bfJ}{\mathbf{J}}
\newcommand{\bfP}{\mathbf{P}}
\newcommand{\bfU}{\mathbf{U}}
\newcommand{\bfx}{\mathbf{x}}
\newcommand{\bfX}{\mathbf{X}}
\newcommand{\bfy}{\mathbf{y}}
\newcommand{\bfZ}{\mathbf{Z}}

\newcommand{\bfone}{\boldsymbol{1}}
\newcommand{\bfzero}{\boldsymbol{0}}
\newcommand{\bfdelta}{\boldsymbol{\delta}}
\newcommand{\bfphi}{\boldsymbol{\varphi}}
\newcommand{\bfPhi}{\boldsymbol{\Phi}}
\newcommand{\bfpsi}{\boldsymbol{\psi}}
\newcommand{\bfPsi}{\boldsymbol{\Psi}}

\newcommand{\calB}{\mathcal{B}}
\newcommand{\calD}{\mathcal{D}}
\newcommand{\calE}{\mathcal{E}}
\newcommand{\calG}{\mathcal{G}}
\newcommand{\calH}{\mathcal{H}}
\newcommand{\calJ}{\mathcal{J}}
\newcommand{\calK}{\mathcal{K}}
\newcommand{\calT}{\mathcal{T}}
\newcommand{\calU}{\mathcal{U}}
\newcommand{\calV}{\mathcal{V}}
\newcommand{\calX}{\mathcal{X}}
\newcommand{\calY}{\mathcal{Y}}

\newcommand{\rmc}{\mathrm{c}}
\newcommand{\rmi}{\mathrm{i}}
\newcommand{\rmO}{\mathrm{O}}
\newcommand{\rmT}{\mathrm{T}}

\newcommand{\Tr}{\operatorname{Tr}}
\newcommand{\tr}{\operatorname{tr}}
\newcommand{\opt}{{\operatorname{opt}}}
\newcommand{\dist}{\operatorname{dist}}
\newcommand{\rank}{\operatorname{rank}}
\newcommand{\BIBD}{{\operatorname{BIBD}}}
\newcommand{\Span}{\operatorname{span}}
\newcommand{\Spark}{{\operatorname{spark}}}

\newcommand{\Fro}{\mathrm{Fro}}

\newcommand{\abs}[1]{|{#1}|}
\newcommand{\bigparen}[1]{\bigl({#1}\bigr)}
\newcommand{\bracket}[1]{[{#1}]}
\newcommand{\bigbracket}[1]{\bigl[{#1}\bigr]}
\newcommand{\set}[1]{\{{#1}\}}
\newcommand{\bigset}[1]{\bigl\{{#1}\bigr\}}
\newcommand{\norm}[1]{\|{#1}\|}
\newcommand{\biggnorm}[1]{\biggl\|{#1}\biggr\|}
\newcommand{\ip}[2]{\langle{#1},{#2}\rangle}

\newtheorem{theorem}{Theorem}[section]
\newtheorem{lemma}[theorem]{Lemma}

\newtheorem{corollary}[theorem]{Corollary}

\theoremstyle{definition}
\newtheorem{definition}[theorem]{Definition}
\newtheorem{example}[theorem]{Example}

\begin{document}
\begin{frontmatter}
\title{Equiangular tight frames that contain regular simplices}

\author[AFIT]{Matthew Fickus}
\ead{Matthew.Fickus@gmail.com}
\author[SDSU]{John Jasper}
\author[Bremen]{Emily J.\ King}
\author[OSU]{Dustin G.\ Mixon}

\address[AFIT]{Department of Mathematics and Statistics, Air Force Institute of Technology, Wright-Patterson AFB, OH 45433}
\address[SDSU]{Department of Mathematics and Statistics, South Dakota State University, Brookings, SD 57007}
\address[Bremen]{Department of Mathematics and Computer Science, University of Bremen, Bremen, Germany 28359}
\address[OSU]{Department of Mathematics, Ohio State University, Columbus, OH 43210}

\begin{abstract}
An equiangular tight frame (ETF) is a type of optimal packing of lines in Euclidean space.
A regular simplex is a special type of ETF in which the number of vectors is one more than the dimension of the space they span.
In this paper, we consider ETFs that contain a regular simplex,
that is, have the property that a subset of its vectors forms a regular simplex.
As we explain, such ETFs are characterized as those that achieve equality in a certain well-known bound from the theory of compressed sensing.
We then consider the so-called binder of such an ETF, namely the set of all regular simplices that it contains.
We provide a new algorithm for computing this binder in terms of products of entries of the ETF's Gram matrix.
In certain circumstances, we show this binder can be used to produce a particularly elegant Naimark complement of the corresponding ETF.
Other times, an ETF is a disjoint union of regular simplices, and we show this leads to a certain type of optimal packing of subspaces known as an equichordal tight fusion frame.
We conclude by considering the extent to which these ideas can be applied to numerous known constructions of ETFs, including harmonic ETFs.
\end{abstract}

\end{frontmatter}

\section{Introduction}

Let $n$ and $d$ be positive integers with $n\geq d$, and let $\bbF$ be either $\bbR$ or $\bbC$.
The \textit{coherence} of a sequence $\set{\bfphi_j}_{j=1}^{n}$ of $n$ nonzero equal norm vectors in a $d$-dimensional Hilbert space $\bbH$ over $\bbF$ is
\begin{equation}
\label{equation.definition of coherence}
\mu:=\max_{j\neq j'}\tfrac{\abs{\ip{\bfphi_j}{\bfphi_{j'}}}}{\norm{\bfphi_j}\norm{\bfphi_{j'}}}.
\end{equation}
In the real case,
each vector $\bfphi_j$ spans a line and $\mu$ is the cosine of the smallest angle between any pair of these lines.
Our work here is motivated by two well-known bounds involving $\mu$.
The first of these is the \textit{Welch bound}~\cite{Welch74},
which is a lower bound on $\mu$ whenever $n\geq d$:
\begin{equation}
\label{equation.Welch bound}
\mu\geq\bigbracket{\tfrac{n-d}{d(n-1)}}^{\frac12}.
\end{equation}
The second bound arises in compressed sensing~\cite{DonohoE02,BrucksteinDE09},
and gives a lower bound on the \textit{spark} of $\set{\bfphi_j}_{j=1}^{n}$,
namely the smallest number of these vectors that are linearly dependent:
\begin{equation}
\label{equation.spark bound}
\Spark\set{\bfphi_j}_{j=1}^{n}\geq\tfrac1{\mu}+1.
\end{equation}

It is well known that \smash{$\set{\bfphi_j}_{j=1}^{n}$} achieves equality in the Welch bound~\eqref{equation.Welch bound} if and only if it is an \textit{equiangular tight frame} (ETF) for $\bbH$, that is, if and only if the value of $\abs{\ip{\bfphi_j}{\bfphi_{j'}}}$ is constant over all $j\neq j'$ (equiangularity) and there exists $a>0$ such that $\sum_{j=1}^{n}\abs{\ip{\bfphi_j}{\bfx}}^2=a\norm{\bfx}^2$ for all $\bfx\in\bbH$ (tightness)~\cite{StrohmerH03}.
This paper focuses on ETFs that achieve equality in~\eqref{equation.spark bound}.
As we shall see, this happens precisely when the ETF contains a \textit{regular simplex},
namely when for some positive integer $s$ there are $s+1$ of the $\bfphi_j$ vectors that form an ETF for an $s$-dimensional subspace of $\bbH$.
Of the few infinite families of ETFs that are known,
a remarkably large proportion of them have this property.
This raises the following fundamental question: in general, to what extent do ETFs contain regular simplices?
Our results here are some first steps towards an answer.

ETFs arise in several applications including waveform design for wireless communication~\cite{StrohmerH03}, compressed sensing~\cite{BajwaCM12,BandeiraFMW13}, quantum information theory~\cite{Zauner99,RenesBSC04} and algebraic coding theory~\cite{JasperMF14}.
They also seem to be rare~\cite{FickusM16}.
With the exception of orthonormal bases and regular simplices,
every known infinite family of ETFs involves some type of combinatorial design.
Real ETFs are equivalent to a subclass of strongly regular graphs (SRGs)~\cite{vanLintS66,Seidel76,HolmesP04,Waldron09},
and such graphs have been actively studied for decades~\cite{Brouwer07,Brouwer15,CorneilM91}.
This equivalence has been partially generalized to the complex setting in various ways,
including approaches that exploit properties of roots of unity~\cite{BodmannPT09,BodmannE10},
distance-regular covers of complete graphs (DRACKNs)~\cite{CoutinkhoGSZ16},
and association schemes~\cite{IversonJM16}.
Necessary integrality conditions on the existence of various types of ETFs are given in~\cite{SustikTDH07}.

Conference matrices, Hadamard matrices, Paley tournaments and quadratic residues are related, and lead to infinite families of ETFs whose \textit{redundancy} $\frac nd$ is either nearly or exactly two~\cite{StrohmerH03,HolmesP04,Renes07,Strohmer08}.
\textit{Harmonic ETFs} and \textit{Steiner ETFs} offer much more freedom in choosing $d$ and $n$.
Harmonic ETFs are equivalent to difference sets in finite abelian groups~\cite{Turyn65,StrohmerH03,XiaZG05,DingF07}.
Steiner ETFs arise from particular types of balanced incomplete block designs (BIBDs)~\cite{GoethalsS70,FickusMT12}.
Recent generalizations of Steiner ETFs have led to new infinite families of ETFs arising from projective planes that contain hyperovals~\cite{FickusMJ16} as well as from Steiner triple systems~\cite{FickusJMP18}.
Another new family arises by generalizing the SRG construction of~\cite{Godsil92} to the complex setting~\cite{FickusJMPW19}.

Many of these known constructions give ETFs that contain a regular simplex,
and thus achieve equality in both~\eqref{equation.Welch bound} and~\eqref{equation.spark bound}.
For example, every Steiner ETF is a disjoint union of regular simplices by construction.
This property is also enjoyed by harmonic ETFs arising from McFarland difference sets,
since it is known that they can be obtained by applying unitary operators to certain Steiner ETFs~\cite{JasperMF14}.
We study ETFs that contain regular simplices in general, and then explore the degree to which the ETFs constructed in~\cite{XiaZG05,DingF07,FickusMJ16,FickusJMP18,FickusJMPW19} have this property.

In particular, in the next section, we establish notation, discuss some known results that we will need later on, and elaborate on the connections between these ideas and compressed sensing.
In Section~3, we show that an ETF achieves equality in~\eqref{equation.spark bound} if and only if it contains a regular simplex (Theorem~\ref{theorem.spark bound equality}),
and give a strong necessary condition on the existence of real ETFs that are \textit{full spark} (Theorem~\ref{theorem.real full spark ETFs}).
In the fourth section, we characterize regular simplices that are contained in an ETF in terms of \textit{triple products} (Theorem~\ref{theorem.triple products}), and then use this idea to develop an algorithm for computing the \textit{binder} of an ETF, namely the set of all simplices it contains.
We build upon these ideas in Section~5, discovering an intimate connection that sometimes arises between an ETF's binder and the phased BIBD ETFs of~\cite{FickusJMPW19}; see Theorems~\ref{theorem.phased binder} and~\ref{theorem.Naimark complement of phased BIBD ETF}.
In Section~6, we give several results about ETFs that happen to be disjoint unions of regular simplices, in particular relating them to \textit{equichordal tight fusion frames} (ECTFFs); see Theorem~\ref{theorem.equichordal}.
In the final section, we then apply these ideas to better understand various existing constructions of ETFs,
in particular showing that certain harmonic ETFs are disjoint unions of regular simplices (Theorems~\ref{theorem.harmonic ETFs} and~\ref{theorem.instances of harmonic ETFs}).

\section{Background}

\subsection{Equiangular tight frames}

Let $\bbF$ be either $\bbR$ or $\bbC$,
and let $\bbH$ be a $d$-dimensional Hilbert space over $\bbF$.
In frame theory, the \textit{synthesis operator} of a finite sequence of vectors $\set{\bfphi_j}_{j=1}^{n}$ in $\bbH$ is $\bfPhi:\bbF^n\rightarrow\bbH$,
$\bfPhi\bfx:=\sum_{j=1}^{n}\bfx(j)\bfphi_j$,
where $\bfx(j)$ denotes the $j$th entry of $\bfx\in\bbF^n$.
Its adjoint is the corresponding \textit{analysis operator} $\bfPhi^*:\bbH\rightarrow\bbF^n$ which satisfies
$(\bfPhi^*\bfy)(j)=\ip{\bfphi_j}{\bfx}$ for all $j=1,\dotsc,n$.
Here and throughout, inner products are taken to be conjugate-linear in its first argument and linear in its second.
Applying the analysis operator to the synthesis operator yields the \textit{Gram matrix} $\bfPhi^*\bfPhi:\bbF^n\rightarrow\bbF^n$,
an $n\times n$ matrix whose $(j,j')$th entry is $(\bfPhi^*\bfPhi)(j,j')=\ip{\bfphi_j}{\bfphi_{j'}}$.
Taking the reverse composition gives the \textit{frame operator} $\bfPhi\bfPhi^*:\bbH\rightarrow\bbH$,
\smash{$\bfPhi\bfPhi^*\bfy=\sum_{j=1}^{n}\ip{\bfphi_j}{\bfy}\bfphi_j$}.
It is well known that sequences $\set{\bfphi_j}_{j=1}^{n}$ in $\bbH$ and \smash{$\set{\hat{\bfphi}_j}_{j=1}^{n}$} in $\hat{\bbH}$ have the same Gram matrix if and only if there exists a unitary operator \smash{$\bfU:\bbH\rightarrow\hat{\bbH}$} such that $\hat{\bfphi}_j=\bfU\bfphi_j$ for all $j=1,\dotsc,n$.

As a slight abuse of notation,
for each $j=1,\dotsc,n$,
we denote the synthesis and analysis operators of the single vector $\set{\bfphi_j}$ as
$\bfphi_j:\bbF\rightarrow\bbH$, $\bfphi_j x=x\bfphi_j$
and $\bfphi_j^*:\bbH\rightarrow\bbF$, $\bfphi_j^*y=\ip{\bfphi_j}{\bfy}$, respectively.
Under this notation, the frame operator of the sequence $\set{\bfphi_j}_{j=1}^{n}$ can be written as $\bfPhi\bfPhi^*=\sum_{j=1}^{n}\bfphi_j^{}\bfphi_j^*$.
We say $\set{\bfphi_j}_{j=1}^{n}$ is \textit{equal norm} if there exists $c>0$ such that $\norm{\bfphi_j}^2=c$ for all $j=1,\dotsc,n$,
and say an equal norm sequence $\set{\bfphi_j}_{j=1}^{n}$ is \textit{equiangular} if there exists $w$ such that $\abs{\ip{\bfphi_j}{\bfphi_{j'}}}^2=w$ for all $j\neq j'$.
We say $\set{\bfphi_j}_{j=1}^{n}$ is a \textit{tight frame} for
$\bbH$ if there exists $a>0$ such that $\bfPhi\bfPhi^*=a\bfI$, that is, if $\bfy=\frac1a\sum_{j=1}^{n}\ip{\bfphi_j}{\bfy}\bfphi_j$ for all $\bfy\in\bbH$;
this requires $n\geq d=\dim(\bbH)$.
If $\set{\bfphi_j}_{j=1}^{n}$ is an equal norm tight frame for $\bbH$,
the constants $c$ and $a$ are related:
$da=\Tr(a\bfI)=\Tr(\bfPhi\bfPhi^*)=\Tr(\bfPhi^*\bfPhi)=\sum_{j=1}^{n}\norm{\bfphi_j}^2=nc$.
Moreover, for any equal norm vectors $\set{\bfphi_j}_{j=1}^{n}$, a direct computation reveals
\begin{equation*}
0
\leq
\norm{\bfPhi\bfPhi^*-\tfrac{nc}d\bfI}_{\Fro}^2
=\sum_{j=1}^{n}\sum_{\substack{j'=1\\j'\neq j}}^{n} \abs{\ip{\bfphi_j}{\bfphi_{j'}}}^2-\tfrac{n(n-d)c^2}{d}
\leq n(n-1)\max_{j\neq j'}\abs{\ip{\bfphi_j}{\bfphi_{j'}}}^2-\tfrac{n(n-d)c^2}{d}.
\end{equation*}
Rearranging this inequality and writing it in terms of the coherence~\eqref{equation.definition of coherence} gives the Welch bound~\eqref{equation.Welch bound}.
Furthermore, $\set{\bfphi_j}_{j=1}^{n}$ achieves equality in~\eqref{equation.Welch bound} if and only if both inequalities above hold with equality, that is, if and only if $\set{\bfphi_j}_{j=1}^{n}$ is both equiangular and a tight frame, namely an ETF.

Sometimes an ETF $\set{\bfphi_j}_{j=1}^{n}$ is most naturally represented by regarding $\bbH=\Span\set{\bfphi_j}_{j=1}^{n}$ as a subspace of some larger Hilbert space \smash{$\hat{\bbH}$}.
(To clarify, as seen by its derivation above, the Welch bound holds for the dimension $d$ of any Hilbert space $\bbH$ that contains $\set{\bfphi_j}_{j=1}^{n}$, and $\set{\bfphi_j}_{j=1}^{n}$ achieves equality in it if and only if $\set{\bfphi_j}_{j=1}^{n}$ is an ETF for $\bbH$.)
There, we can regard \smash{$\hat{\bbH}$} as the codomain of $\bfPhi$,
leading to an extension \smash{$\bfPhi^*:\hat{\bbH}\rightarrow\bbF^n$} of the analysis operator.
Here, we have the following characterization of when $\set{\bfphi_j}_{j=1}^{n}$ is a tight frame for its span:
\begin{lemma}[Lemma~1 of~\cite{FickusMJ16}]
\label{lemma:tight frame for its span}
For any $a>0$, the following are equivalent:
\begin{enumerate}
\renewcommand{\labelenumi}{(\roman{enumi})}
\item $\set{\bfphi_j}_{j=1}^n$ is an $a$-tight frame for its span, that is, $\bfPhi\bfPhi^*\bfy=a\bfy$ for all $\bfy\in\Span\set{\bfphi_j}_{j=1}^{n}$,
\item $\bfPhi\bfPhi^*\bfPhi = a\bfPhi$,
\item $(\bfPhi\bfPhi^*)^2 = a\bfPhi\bfPhi^*$,
\item $(\bfPhi^*\bfPhi)^2 = a\bfPhi^*\bfPhi$.
\end{enumerate}
\end{lemma}

In particular, $\set{\bfphi_j}_{j=1}^{n}$ is an ETF for its span if and only if it is equiangular and these properties hold.
In this case, the dimension $d$ of its span can be inferred from $c=\norm{\bfphi_j}^2$ and $a=\frac{nc}d$, namely
\begin{equation}
\label{equation.dimension of ETF span}
d=\tfrac{nc}{a}.
\end{equation}
In the special case where \smash{$\hat{\bbH}=\bbF^m$}, $\bfPhi$ is naturally represented by the $m\times n$ matrix whose $j$th column is $\bfphi_j$, and $\bfPhi^*$ is its $n\times m$ conjugate transpose.

\subsection{Regular simplices and Naimark complements}

Since any subset of equiangular vectors is automatically equiangular, it is natural to ask whether a given ETF contains a subset that is an ETF for its span.
This leads to the construction of~\cite{FickusMJ16}, for example.
In this paper, we are chiefly concerned with ETFs that contain a very simple type of ETF for its span known as a regular simplex:

\begin{definition}
\label{definition.simplex}
For a positive integer $s$, a sequence of $s+1$ vectors $\set{\bfphi_j}_{j=1}^{s+1}$ in a Hilbert space $\bbH$ is a \textit{regular $s$-simplex} if it is an ETF for an $s$-dimensional subspace of $\bbH$, namely when it is an ETF for its span and this span has dimension $s$.
\end{definition}

In the case where a regular $s$-simplex appears as a subset of a larger ETF $\set{\bfphi_j}_{j=1}^{n}$ we say that $\set{\bfphi_j}_{j=1}^{n}$ \textit{contains} the simplex.
Letting $n=s+1$ and $d=s$ in~\eqref{equation.Welch bound} gives that the coherence of a regular $s$-simplex is \smash{$\frac1s$}.
Since the coherence of a subset of $\set{\bfphi_j}_{j=1}^{n}$ equals that of $\set{\bfphi_j}_{j=1}^{n}$ itself,
we see that if an ETF $\set{\bfphi_j}_{j=1}^{n}$ for a $d$-dimensional space $\bbH$ contains a regular $s$-simplex, then $s$ is necessarily the reciprocal of the Welch bound, namely:
\begin{equation}
\label{equation.inverse Welch bound}
s
=\bigbracket{\tfrac{d(n-1)}{n-d}}^{\frac12}.
\end{equation}
In particular, an ETF can only contain a simplex if its coherence is the reciprocal of an integer.
This immediately implies that some ETFs do not contain a regular simplex,
including certain real ETFs with $n=2d$ as well as many complex ETFs.
For real ETFs with $n\neq 2d$ and $1<d<n-1$, it is known that its inverse Welch bound is necessarily an odd integer~\cite{SustikTDH07}, and so such ETFs already meet this necessary condition.
It turns out this necessary condition is not sufficient.
For example, we will give an example of a $9$-vector ETF for \smash{$\bbC^6$} that contains no simplices, despite the fact that $s=4$ is an integer.

Regular simplices exist for any positive integer $s$.
To see this, and to help find any regular simplices that are contained in a given ETF,
it helps to discuss the notion of the \textit{Naimark complement} of a tight frame.
Specifically, if $\set{\bfphi_j}_{j=1}^{n}$ is any $a$-tight frame for its $d$-dimensional span where $d<n$,
then its $n\times n$ Gram matrix $\bfPhi^*\bfPhi$ has eigenvalues $a$ and $0$ with multiplicities $d$ and $n-d$, respectively.
As such, $a\bfI-\bfPhi^*\bfPhi$ has eigenvalues $a$ and $0$ with multiplicities $n-d$ and $d$, respectively,
meaning there exist vectors $\set{\bfpsi_j}_{j=1}^{n}$ in an $(n-d)$-dimensional Hilbert space such that $\bfPsi^*\bfPsi=a\bfI-\bfPhi^*\bfPhi$.
Being defined by their Gram matrix, the vectors $\set{\bfpsi_j}_{j=1}^{n}$ are only unique up to unitary transformations.
Particular examples can be constructed, for example, by representing $\bfPhi$ as a $d\times n$ matrix and choosing $\bfPsi$ to be an $(n-d)\times n$ matrix whose rows complete those of $\bfPhi$ to an equal norm orthogonal basis for $\bbF^n$.
If $\set{\bfphi_j}_{j=1}^{n}$ and $\set{\bfpsi_j}_{j=1}^{n}$ are Naimark complements then \begin{equation}
\label{equation.cross frame operator of Naimark complements}
\bfPhi\bfPsi^*=\bfzero.
\end{equation}
Indeed, Lemma~\ref{lemma:tight frame for its span} gives
$(\bfPsi\bfPhi^*)^*(\bfPsi\bfPhi^*)
=\bfPhi\bfPsi^*\bfPsi\bfPhi^*
=\bfPhi(a\bfI-\bfPhi^*\bfPhi)\bfPhi^*
=a\bfPhi\bfPhi^*-(\bfPhi\bfPhi^*)^2
=\bfzero$,
implying $\norm{\bfPsi\bfPhi^*}_{\Fro}^2=\Tr(\bfzero)=0$,
and so $\bfPsi\bfPhi^*$ and its adjoint $\bfPhi\bfPsi^*$ are zero.

When $\set{\bfphi_j}_{j=1}^{n}$ is an ETF for its span,
so is each one of its Naimark complements:
we have $(\bfPsi^*\bfPsi)^2=(a\bfI-\bfPhi^*\bfPhi)^2=a(a\bfI-\bfPhi^*\bfPhi)=a\bfPsi^*\bfPsi$, $\norm{\bfpsi_j}^2=a-\norm{\bfphi_j}^2=a-c$ for all $j$, and $\ip{\bfpsi_j}{\bfpsi_{j'}}=-\ip{\bfphi_j}{\bfphi_{j'}}$ for all $j\neq j'$.
In particular, for any positive integer $s$, a regular $s$-simplex exists:
any sequence $\set{c_j}_{j=1}^{s+1}$ of $s+1$ unimodular scalars is an ETF for $\bbF^1$, and so any one of its Naimark complements is an $(s+1)$-vector ETF for an $s$-dimensional space.
Conversely,
if \smash{$\set{\bfphi_j}_{j=1}^{s+1}$} is a regular simplex then any one of its Naimark complements is an ETF for $\bbF^1$.

\subsection{The restricted isometry property}

In compressed sensing, a sequence of vectors $\set{\bfphi_j}_{j=1}^{n}$ in $\bbH$ is said to have the \textit{restricted isometry property} (RIP) for a given positive integer $k$ and $\delta\in[0,1)$ if
\begin{equation}
\label{equation.RIP}
(1-\delta)\norm{\bfx}^2
\leq\norm{\bfPhi\bfx}^2
\leq(1+\delta)\norm{\bfx}^2
\end{equation}
for all \textit{$k$-sparse} vectors $\bfx\in\bbF^n$,
that is, for all $\bfx\in\bbF^n$ with at most $k$ nonzero entries.
Equivalently, letting $\calK$ be a subset of $[n]:=\set{1,\dotsc,n}$ and letting $\bfPhi_\calK:\bbF^{\calK}\rightarrow\bbH$ denote the synthesis operator of $\set{\bfphi_j}_{j\in\calK}$,
we have that $\set{\bfphi_j}_{j=1}^{n}$ is $(k,\delta)$-RIP if $\norm{\bfPhi_\calK^*\bfPhi_\calK^{}-\bfI}_2\leq\delta$
for all $k$-element subsets $\calK$ of $[n]$.
Note that since $\delta<1$, each matrix $\bfPhi_\calK^*\bfPhi_\calK^{}$ is invertible,
implying any $k$ vectors $\set{\bfphi_j}_{j\in\calK}$ are linearly independent,
and so $\Spark\set{\bfphi_j}_{j=1}^{n}\geq k+1$.
From this perspective, RIP is a numerically robust notion of spark: it means that every $k$ of the vectors $\set{\bfphi_j}_{j\in\calK}$ are $\delta$-close to being orthonormal.

Coherence is related to RIP in several ways.
Most simply, if $\set{\bfphi_j}_{j=1}^{n}$ is unit norm and $(k,\delta)$-RIP for some $k\geq2$, then it is also $(2,\delta)$-RIP, meaning
\begin{equation*}
\delta
\geq\max_{\#(\calK)=2}\norm{\bfPhi_\calK^*\bfPhi_\calK^{}-\bfI}_2
=\max_{j\neq j'}\left\|\left[\begin{array}{cc}
0&\ip{\bfphi_j}{\bfphi_{j'}}\\
\ip{\bfphi_{j'}}{\bfphi_j}&0
\end{array}\right]\right\|_2
=\max_{j\neq j'}\abs{\ip{\bfphi_j}{\bfphi_{j'}}}
=\mu.
\end{equation*}
That is, the coherence $\mu$ is a lower bound on $\delta$.
For a corresponding upper bound, note that by the Gershgorin circle theorem,
the optimal $\delta$ for any given $k$ and unit vectors $\set{\bfphi_j}_{j=1}^{n}$ satisfies
\begin{equation*}
\delta_{\opt}
:=\max_{\#(\calK)=k}\norm{\bfPhi_\calK^*\bfPhi_\calK^{}-\bfI}_2
\leq\max_{\#(\calK)=k}\max_{j\in\calK}
\sum_{\substack{j'\in\calK\\j'\neq j}}\abs{\ip{\bfphi_{j'}}{\bfphi_j}}
\leq (k-1)\mu.
\end{equation*}
In particular, $\set{\bfphi_j}_{j=1}^{n}$ is $(k,\delta)$-RIP for some $\delta\in[0,1)$ if $(k-1)\mu<1$, that is, if \smash{$k<\frac1{\mu}+1$}.
This fact leads to~\eqref{equation.spark bound}:
we have $\Spark\set{\bfphi_j}_{j=1}^{n}\geq k+1$ for all \smash{$k<\frac1{\mu}+1$},
and so \smash{$\Spark\set{\bfphi_j}_{j=1}^{n}\geq\frac1{\mu}+1$}.

The fact that $\mu\leq\delta_{\opt}\leq(k-1)\mu$ suggests that ETFs are RIP to a high degree.
However, this is only true to a certain extent.
To elaborate, when $n\geq 2d$ we have \smash{$\frac{n-1}{n-d}\leq 2$} and the Welch bound~\eqref{equation.Welch bound} implies
\smash{$\frac1{\mu}+1\leq(2d)^{\frac12}+1$}.
Thus, the sufficient condition \smash{$k<\frac1{\mu}+1$} only holds for \smash{$k=\rmO(d^{\frac12})$}.
In contrast, certain random matrices are RIP with high probability provided
$k=\rmO(d/\operatorname{polylog}(n))$.
This is great in compressed sensing applications since it means the number of measurements $d$ used to sense a $k$-sparse signal $\bfx\in\bbF^n$ can grow almost-linearly with $k$, instead of quadratically.

The fact that coherence-based guarantees of RIP only permit \smash{$k=\rmO(d^{\frac12})$} is known as the \textit{square-root bottleneck}.
To date, \cite{BourgainDFKK11} gives the only deterministic construction of a sequence $\set{\bfphi_j}_{j=1}^{n}$ that is guaranteed to be RIP with \smash{$k=\rmO(d^{\frac12+\varepsilon})$} for some $\varepsilon>0$,
and this $\varepsilon$ seems extremely small~\cite{Mixon13}.
Whether any ETF can be RIP to a degree that rivals random matrices is a fundamental open problem~\cite{BandeiraFMW13}.
However, it is known that certain ETFs are incapable of breaking the square-root bottleneck:
any ETF that achieves equality in~\eqref{equation.spark bound} contains \smash{$\frac1\mu+1$} vectors that are linearly dependent,
and \smash{$\frac1\mu=\rmO(d^{\frac12})$} for any infinite family of ETFs for which $\frac nd$ is bounded away from $1$.
This includes all Steiner ETFs~\cite{FickusMT12}.

In Theorem~\ref{theorem.spark bound equality} below, we show that an ETF achieves equality in~\eqref{equation.spark bound} if and only if it contains a regular simplex.
As such, from an RIP perspective, these ETFs are the ``worst of the best," having the worst possible spark of all $\set{\bfphi_j}_{j=1}^{n}$ in $\bbH$ that have the best possible coherence.
We study them here to help narrow down the search for ETFs with better RIP properties,
and to gain more insight into ETFs in general, since even ETFs that achieve equality in~\eqref{equation.spark bound} arise in several applications, including coherence-based compressed sensing~\cite{Tropp04}.

\section{The spark of an equiangular tight frame}

In this section, we explore the extreme values that the spark of an ETF can attain.
We begin by considering the lower bound on the spark given in~\eqref{equation.spark bound}:

\begin{theorem}
\label{theorem.spark bound equality}
If $\set{\bfphi_j}_{j=1}^{n}$ is an ETF for a Hilbert space $\bbH$ of dimension $d$,
then a subset $\set{\bfphi_j}_{j\in\calK}$
of this ETF is a regular simplex if and only if it consists of $s+1$ linearly dependent vectors where $s$ is the inverse Welch bound~\eqref{equation.inverse Welch bound}.

As a corollary, $\set{\bfphi_j}_{j=1}^{n}$ achieves equality in~\eqref{equation.spark bound} if and only if it contains a regular simplex.
\end{theorem}

\begin{proof}
($\Rightarrow$)
For any positive integer $s$,
any regular $s$-simplex in \smash{$\set{\bfphi_j}_{j=1}^{n}$} may without loss of generality be written as \smash{$\set{\bfphi_j}_{j=1}^{s+1}$}.
Having an $s$-dimensional span by assumption, the vectors \smash{$\set{\bfphi_j}_{j=1}^{s+1}$} are linearly dependent.
Moreover, since \smash{$\set{\bfphi_j}_{j=1}^{n}$} is equiangular, its coherence equals that of \smash{$\set{\bfphi_j}_{j=1}^{s+1}$}.
Since both sequences are ETFs---for spaces of dimension $d$ and $s$, respectively---this means their Welch bounds are equal:
\begin{equation*}
\bigbracket{\tfrac{n-d}{d(n-1)}}^{\frac12}
=\bigset{\tfrac{(s+1)-s}{s[(s+1)-1]}}^{\frac12}
=\tfrac1s.
\end{equation*}
($\Leftarrow$)
Let $s$ be given by~\eqref{equation.inverse Welch bound} and without loss of generality write any $s+1$ linearly dependent vectors from \smash{$\set{\bfphi_j}_{j=1}^{n}$} as \smash{$\set{\bfphi_j}_{j=1}^{s+1}$}.
Being linearly dependent, \smash{$\set{\bfphi_j}_{j=1}^{s+1}$} is contained in some $s$-dimensional subspace $\bbH_0$ of $\bbH$.
Being a subset of the ETF \smash{$\set{\bfphi_j}_{j=1}^{n}$},  \smash{$\set{\bfphi_j}_{j=1}^{s+1}$} has coherence
\smash{$\bracket{\tfrac{n-d}{d(n-1)}}^{\frac12}=\tfrac1s$}.
In particular, \smash{$\set{\bfphi_j}_{j=1}^{n}$} meets the Welch bound for any $s+1$ equal norm vectors in $\bbH_0$, meaning it is an ETF for $\bbH_0$, and so is a regular simplex.

For the second conclusion, note that if~\eqref{equation.spark bound} holds then \smash{$\set{\bfphi_j}_{j=1}^{n}$} has spark $s+1$ where $s$ is given by~\eqref{equation.inverse Welch bound}.
The corresponding $s+1$ linearly dependent vectors are a regular simplex.
Conversely, any regular simplex contained in \smash{$\set{\bfphi_j}_{j=1}^{n}$} consists of $s+1$ linearly dependent vectors where $s$ is given by~\eqref{equation.inverse Welch bound}.
Thus, \smash{$\Spark\set{\bfphi_j}_{j=1}^{n}\leq s+1$}.
At the same time, \eqref{equation.spark bound} and~\eqref{equation.inverse Welch bound} give \smash{$\Spark\set{\bfphi_j}_{j=1}^{n}\geq s+1$},
implying~\eqref{equation.spark bound} holds with equality.
\end{proof}

Combining this result with~\eqref{equation.spark bound}, we see that when an ETF contains a regular simplex,
the simplices it contains are precisely the smallest linearly dependent subsets of it,
namely the smallest \textit{circuits} of the corresponding \textit{matroid}~\cite{Cahill09,King15}.
For a small example of a nontrivial ETF that contains a regular simplex, consider the columns $\set{\bfphi_j}_{j=1}^{10}$ of the matrix
\begin{equation}
\label{equation.5x10 ETF}
\bfPhi
=\left[\begin{array}{rrrrrrrrrr}
 1& 1& 1& 1& 1& 1& 1& 1& 1& 1\\
 1& 1& 1& 1&-1&-1&-1&-1&-1&-1\\
 1&-1&-1&-1& 1& 1& 1&-1&-1&-1\\
-1& 1&-1&-1& 1&-1&-1& 1& 1&-1\\
-1&-1& 1&-1&-1& 1&-1& 1&-1& 1\\
-1&-1&-1& 1&-1&-1& 1&-1& 1& 1
\end{array}\right].
\end{equation}
These $10$ columns form an ETF for the $5$-dimensional orthogonal complement of the all-ones vector in $\bbR^6$,
since they clearly lie in that space and their coherence achieves the corresponding Welch bound $\frac13$.
Here, Theorem~\ref{theorem.spark bound equality} gives that a subset of $\set{\bfphi_j}_{j=1}^{10}$ is a regular simplex if and only if it consists of $4$ linearly dependent vectors.
These subsets exist: we have $\bfphi_1-\bfphi_2-\bfphi_6+\bfphi_8=\bfzero$ for example, meaning $\set{\bfphi_1,\bfphi_2,\bfphi_6,\bfphi_8}$ is a regular simplex, that is, a $4$-vector ETF for its $3$-dimensional span.
As such, $\Spark\set{\bfphi_j}_{j=1}^{10}=4$.

Having characterized when the spark of an ETF achieves one extreme, we turn to the other:
if $\set{\bfphi_j}_{j=1}^{n}$ is any sequence of $n$ vectors in a Hilbert space $\bbH$ of dimension $d\leq n-1$,
then any $d+1$ of these vectors are linearly dependent and so $\Spark\set{\bfphi_j}_{j=1}^{n}\leq d+1$.
Moreover, in this setting $\set{\bfphi_j}_{j=1}^{n}$ achieves equality in this upper bound if and only if it is \textit{full spark}, namely when every $d$-vector subset of $\set{\bfphi_j}_{j=1}^{n}$ forms a basis for $\bbH$.
Every orthonormal basis for $\bbH$ is a full spark ETF for it.
Moreover, by~\eqref{equation.spark bound}, every regular simplex is a full spark ETF for its span.
Certain harmonic ETFs are also known to be full spark, including all of those arising from difference sets in groups of prime order~\cite{AlexeevCM12}.
For any prime $p\equiv 3\bmod 4$ for example, we can form a $p$-vector full spark ETF for \smash{$\bbC^{\frac{p-1}2}$} by extracting the $\frac{p-1}2$ rows of a $p\times p$ discrete Fourier transform matrix that correspond to quadratic residues in $\bbZ_p$.
Apart from these families, few examples of full spark ETFs are known,
and they may be rare.
In fact, we can prove this is indeed the case in the real setting:

\begin{theorem}
\label{theorem.real full spark ETFs}
If $\set{\bfphi_j}_{j=1}^{n}$ is a full spark ETF for $\bbR^d$ and $1<d<n-1$ then $n=2d$.
\end{theorem}

\begin{proof}
Let $\set{\bfphi_j}_{j=1}^{n}$ be a full spark ETF for $\bbR^d$ where $1<d<n-1$,
and assume to the contrary that $n\neq 2d$.
It is known that a tight frame is full spark if and only if its Naimark complements are too: see Theorem~4 of~\cite{AlexeevCM12}, its proof, and subsequent comments.
As such, we may assume without loss of generality that $n>2d$.
In particular, we have \smash{$\abs{\ip{\bfphi_j}{\bfphi_{2d}}}=\abs{\ip{\bfphi_j}{\bfphi_{2d+1}}}$} for all $j=1,\dotsc,2d-1$.
Moreover, $\bfphi_{2d}\neq\pm\bfphi_{2d+1}$ since $\set{\bfphi_j}_{j=1}^{n}$ is an ETF for $\bbR^d$ with $d>1$.
Together, these facts imply that \smash{$\set{\bfphi_j}_{j=1}^{2d-1}$} is incapable of \textit{phase retrieval}.
And, by the characterization of phase retrieval in $\bbR^d$ given in~\cite{BalanCE06}, this means
$\set{\bfphi_j}_{j=1}^{n}$ fails to have the \textit{complement property}~\cite{BandeiraCMN14}.
That is, the $2d-1$ vectors \smash{$\set{\bfphi_j}_{j=1}^{2d-1}$} can be partitioned into two subsets with the property that neither subset spans $\bbR^d$.
In particular, the larger of these two sets contains at least $d$ vectors but does not span $\bbR^d$, implying $\Spark\set{\bfphi_j}_{j=1}^{n}\leq d$, a contradiction.
\end{proof}

Real full spark ETFs with $n=2d$ exist.
For example,
for the well-known $6$-vector ETF for $\bbR^3$,
\eqref{equation.spark bound} implies $\Spark\set{\bfphi_j}_{j=1}^{6}\geq\sqrt{5}+1$ and so $\Spark\set{\bfphi_j}_{j=1}^{6}=4$.
That said, even the strong necessary condition of Theorem~\ref{theorem.real full spark ETFs} is not sufficient.
Indeed, the columns $\set{\bfphi_j}_{j=1}^{10}$ of~\eqref{equation.5x10 ETF} form a real $10$-vector ETF for a $5$-dimensional space, and we have already seen that $\Spark\set{\bfphi_j}_{j=1}^{10}=4<5$.
For the remainder of this paper, we focus on ETFs that contain a regular simplex,
and leave a deeper investigation of full spark ETFs for future research.

\section{Finding the regular simplices contained in an equiangular tight frame}

In this section we focus on the following problem:
is there a computationally reasonable way of finding any and all regular simplices contained in a given ETF?
In light of Theorem~\ref{theorem.spark bound equality},
this equates to letting $s$ be the reciprocal of the Welch bound~\eqref{equation.inverse Welch bound}, and finding all subsets $\calK$ of $s+1$ indices in $[n]$ such that $\set{\bfphi_j}_{j\in\calK}$ is linearly dependent.
We give this set a name:

\begin{definition}
The \textit{binder} of an ETF $\set{\bfphi_j}_{j=1}^{n}$ is the set of all subsets $\calK$ of $[n]$ such that $\set{\bfphi_j}_{j\in\calK}$ is a regular simplex.
\end{definition}

For the ETF given in~\eqref{equation.5x10 ETF} for example,
we want to find all linearly dependent $4$-vector subsets of $\set{\bfphi_j}_{j=1}^{10}$.
In such a small example, this can be done by simply checking all \smash{$\binom{10}{4}=210$} subsets of this size.
Doing such reveals $15$ simplices, each indexed by a $4$-element subset of $\set{1,\dotsc,10}$, namely the binder:
\begin{equation}
\label{equation.5x10 binder}
\begin{gathered}
\set{1,2,6,8},
\set{1,2,7,9},
\set{1,3,5,8},
\set{1,3,7,10},
\set{1,4,5,9},
\set{1,4,6,10},
\set{2,3,5,6},
\set{2,3,9,10},\\
\set{2,4,5,7},
\set{2,4,8,10},
\set{3,4,6,7},
\set{3,4,8,9},
\set{5,6,9,10},
\set{5,7,8,10},
\set{6,7,8,9}.
\end{gathered}
\end{equation}
However, this direct approach is computationally intractable for even modest values of $n$ and $d$.
For example, a $96$-vector ETF for $\bbC^{76}$ has \smash{$\binom{n}{s+1}=\binom{96}{20}\approx 2.16\times 10^{20}$} subsets of size $s+1=20$.
Our numerical experiments indicate that it would take a current desktop computer billions of years to check them all for linear dependence.
We now discuss an alternative approach that, at least for certain ETFs of this size, computes their binders in a matter of hours on the same hardware.

The main idea is that given the Gram matrix of an ETF,
a principal submatrix of it that arises from a regular simplex is very special:
since that regular simplex is a Naimark complement of an ETF for $\bbF^1$,
this principal submatrix is closely related to a rank-one matrix.
As we show below, this means they can be simply characterized in terms of the \textit{triple products} of the ETF $\set{\bfphi_j}_{j=1}^{n}$, namely the values
$\ip{\bfphi_{j_1}}{\bfphi_{j_2}}\ip{\bfphi_{j_2}}{\bfphi_{j_3}}\ip{\bfphi_{j_3}}{\bfphi_{j_1}}$
over all distinct $j_1,j_2,j_3\in[n]$.

Triple products have recently been used to help characterize frames up to \textit{projective unitary equivalence}~\cite{ChienW16}.
Also see the related concept of \textit{flux} in~\cite{BodmannE10b}.
To elaborate, given any ETF $\set{\bfphi_j}_{j=1}^{n}$ for $\bbH$,
any unimodular scalars $\set{c_j}_{j=1}^{n}$,
unitary operator \smash{$\bfU:\bbH\rightarrow\hat{\bbH}$}, and permutation $\sigma:[n]\rightarrow[n]$,
it is straightforward to show that $\set{c_j\bfU\bfphi_{\sigma(j)}}_{j=1}^{n}$ is an ETF for \smash{$\hat{\bbH}$}.
We say two ETFs are \textit{equivalent} if they are related in this way.
Equivalent ETFs have much in common.
For example, they have the same spark and are RIP to the same degree.
(Other important properties, such as \textit{flatness}, are not necessarily preserved by this equivalence~\cite{JasperMF14,FickusMJ16}.)

Given two ETFs with the same $d$ and $n$ parameters, it is often not easy to determine whether they are equivalent.
To be clear, the ambiguity due to $\bfU$ can be removed by instead working with the ETFs' Gram matrices:
the true problem is given two self-adjoint $n\times n$ matrices $\bfG$ and $\hat{\bfG}$ with constant diagonal entries, off-diagonal entries of constant modulus,
and \smash{$\bfG^2=a\bfG$} and \smash{$\hat{\bfG}^2=a\hat{\bfG}$} for some $a>0$,
do there exist unimodular scalars $\set{c_j}_{j=1}^{n}$ and a permutation $\sigma$ such that
\smash{$\hat{\bfG}(j,j')=\overline{c_j}c_{j'}\bfG(\sigma(j),\sigma(j'))$} for all $j,j'=1,\dotsc,n$?
In the special case where $\sigma$ is the identity, the two ETFs are said to be \textit{projective unitary equivalent}~\cite{ChienW16},
and this question can be quickly answered by asking whether the matrix with entries $\hat{\bfG}(j,j')/\bfG(j,j')$ is rank one.
On the other hand, even in the special case where the ETFs are real and $c_j=1$ for all $j$,
determining equivalence reduces to the famously difficult \textit{graph isomorphism problem}.

Triple products are useful when determining when two ETFs are equivalent, since they are invariant with respect to both unimodular scalars and unitary transformations:
\begin{multline}
\label{equation.triple product invariance}
\ip{c_{j_1}\bfU\bfphi_{j_1}}{c_{j_2}\bfU\bfphi_{j_2}}\ip{c_{j_2}\bfU\bfphi_{j_2}}{c_{j_3}\bfU\bfphi_{j_3}}\ip{c_{j_3}\bfU\bfphi_{j_3}}{c_{j_1}\bfU\bfphi_{j_1}}\\
=\overline{c_{j_1}}c_{j_2}\overline{c_{j_2}}c_{j_3}\overline{c_{j_3}}c_{j_1}\ip{\bfphi_{j_1}}{\bfphi_{j_2}}\ip{\bfphi_{j_2}}{\bfphi_{j_3}}\ip{\bfphi_{j_3}}{\bfphi_{j_1}}
=\ip{\bfphi_{j_1}}{\bfphi_{j_2}}\ip{\bfphi_{j_2}}{\bfphi_{j_3}}\ip{\bfphi_{j_3}}{\bfphi_{j_1}}.
\end{multline}
We now use them to characterize the regular simplices contained in an ETF.
The main idea is that a self-adjoint matrix with unimodular entries has rank one if and only if all of its triple products have value $1$.
As such, a regular simplex---the Naimark complement of a rank-one matrix---is characterized by having triple products of value $(-1)^3=-1$:

\begin{theorem}
\label{theorem.triple products}
Let $\set{\bfphi_j}_{j=1}^{n}$ be an equiangular tight frame for a space of dimension $d<n$,
without loss of generality scaled so that
\smash{$\norm{\bfphi_j}^2=s$} where $s$ is the inverse Welch bound~\eqref{equation.inverse Welch bound}.
Then for any $(s+1)$-element subset $\calK$ of $[n]$, $\set{\bfphi_j}_{j\in\calK}$ is a regular simplex if and only if
\begin{equation}
\label{equation.triple product condition}
\ip{\bfphi_{j_1}}{\bfphi_{j_2}}\ip{\bfphi_{j_2}}{\bfphi_{j_3}}\ip{\bfphi_{j_3}}{\bfphi_{j_1}}=-1
\end{equation}
for all pairwise-distinct $j_1,j_2,j_3\in\calK$.
Moreover,
in this case a sequence of scalars $\set{c_j}_{j\in\calK}$ is a Naimark complement for $\set{\bfphi_j}_{j\in\calK}$ if and only if
\begin{equation}
\label{equation.regular simplex Naimark complement}
\sum_{j=1}^{s+1}\overline{c_j}\bfphi_j=\bfzero
\quad
\text{and}
\quad
\abs{c_j}=1,\ j=1,\dotsc,n.
\end{equation}
\end{theorem}

\begin{proof}
Note this scaling ensures that $\abs{\ip{\bfphi_j}{\bfphi_{j'}}}=\mu\norm{\bfphi_j}\norm{\bfphi_{j'}}=1$ for any $j\neq j'$.
Fix any subset $\calK$ of $[n]$ of cardinality $s+1$.

($\Rightarrow$)
Assume \smash{$\set{\bfphi_j}_{j\in\calK}$} is a regular simplex for $\bbH$,
namely an ETF for an $s$-dimensional subspace $\bbH_0$ of $\bbH$.
By~\eqref{equation.dimension of ETF span},
its tight frame constant is necessarily \smash{$\frac{s+1}{\dim(\bbH_0)}\norm{\bfphi_j}^2=\frac{s+1}{s}s=s+1$}.
As such, taking a Naimark complement of $\set{\bfphi_j}_{j\in\calK}$ yields a sequence $\set{c_j}_{j\in\calK}$ of scalars whose $1\times(s+1)$ synthesis operator $\bfC$ satisfies $\bfC^*\bfC=(s+1)\bfI-\bfPhi_{\calK}^*\bfPhi_{\calK}^{}$.
This means $\abs{c_j}^2=(s+1)-\norm{\bfphi_j}^2=1$ for all $j\in\calK$ and $\ip{\bfphi_j}{\bfphi_{j'}}=-\overline{c_j}c_{j'}$ for all $j,j'\in\calK$, $j\neq j'$.
Thus, for all pairwise-distinct $j_1,j_2,j_3\in\calK$,
\begin{equation*}
\ip{\bfphi_{j_1}}{\bfphi_{j_2}}\ip{\bfphi_{j_2}}{\bfphi_{j_3}}\ip{\bfphi_{j_3}}{\bfphi_{j_1}}
=(-\overline{c_{j_1}}c_{j_2})(-\overline{c_{j_2}}c_{j_3})(-\overline{c_{j_3}}c_{j_1})
=-\abs{c_{j_1}}^2\abs{c_{j_2}}^2\abs{c_{j_3}}^2
=-1.
\end{equation*}
Note that in this case~\eqref{equation.cross frame operator of Naimark complements} also gives
\smash{$\sum_{j=1}^{s+1}\overline{c_j}\bfphi_j
=\bfPhi\bfC^*
=\bfzero$}, and so any Naimark complement of any regular simplex \smash{$\set{\bfphi_j}_{j\in\calK}$} indeed satisfies
\eqref{equation.regular simplex Naimark complement}.

($\Leftarrow$)
Assume \eqref{equation.triple product condition} holds for all pairwise-distinct $j_1,j_2,j_3\in\calK$.
Fix any $j_3\in\calK$, and let $\hat{\bfC}$ be the $1\times(s+1)$ synthesis operator of the sequence \smash{$\set{\hat{c}_j}_{j\in\calK}$} of scalars defined by $\hat{c}_{j_3}:=1$ and $\hat{c}_j:=-\ip{\bfphi_{j_3}}{\bfphi_j}$ for all $j\in\calK$, $j\neq j_3$.
Being unimodular, \smash{$\set{\hat{c}_j}_{j\in\calK}$} is an ETF for $\bbF^1$ with tight frame constant $s+1$.
We claim that
\begin{equation}
\label{equation.proof of triple products 1}
\bfPhi_{\calK}^*\bfPhi_{\calK}^{}=(s+1)\bfI-\hat{\bfC}^*\hat{\bfC}.
\end{equation}
Indeed, these two matrices have equal diagonal entries since
$\norm{\bfphi_j}^2=s=(s+1)-\abs{\hat{c}_j}^2$ for any $j\in\calK$.
Meanwhile, for any $j_1,j_2\in\calK$, $j_1\neq j_2$,
our assumption~\eqref{equation.triple product condition} gives
\begin{equation*}
\overline{\hat{c}_{j_1}}\hat{c}_{j_2}
=\left\{\begin{array}{cc}
-\ip{\bfphi_{j_3}}{\bfphi_{j_2}},&j_1=j_3\medskip\\
\overline{-\ip{\bfphi_{j_3}}{\bfphi_{j_1}}},&j_2=j_3\medskip\\
(-1)^2\overline{\ip{\bfphi_{j_3}}{\bfphi_{j_1}}}\ip{\bfphi_{j_3}}{\bfphi_{j_2}},&\text{else}
\end{array}\right\}
=-\ip{\bfphi_{j_1}}{\bfphi_{j_2}}.
\end{equation*}
Having \eqref{equation.proof of triple products 1}, we see that \smash{$\set{\bfphi_j}_{j\in\calK}$} is a Naimark complement for \smash{$\set{\hat{c}_j}_{j\in\calK}$}, and so is an $(s+1)$-vector ETF for a space of dimension $s$, namely a regular simplex.

Again assuming $\set{\bfphi_j}_{j\in\calK}$ is a regular simplex,
all that remains to be shown is that any sequence of scalars $\set{c_j}_{j\in\calK}$ that satisfies~\eqref{equation.regular simplex Naimark complement} is a Naimark complement of it.
From above, we know the aforementioned sequence
\smash{$\set{\hat{c}_j}_{j\in\calK}$} is a Naimark complement for $\set{\bfphi_j}_{j\in\calK}$ and so satisfies~\eqref{equation.regular simplex Naimark complement}.
If $\set{c_j}_{j\in\calK}$ also satisfies~\eqref{equation.regular simplex Naimark complement},
then both sequences lie in the null space of the synthesis operator $\bfPhi_\calK$ of the regular simplex.
Since this space has dimension $(s+1)-\rank\set{\bfphi_j}_{j\in\calK}=(s+1)-s=1$, this implies that $\set{\hat{c}_j}_{j\in\calK}$ and $\set{c_j}_{j\in\calK}$ are scalar multiples of each other.
Since both sequences are unimodular, there in particular exists a unimodular scalar $z$ such that $c_j=z\hat{c}_j$ for all $j\in\calK$.
Thus, $\bfC^*\bfC=\hat{\bfC}^*\hat{\bfC}=(s+1)\bfI-\bfPhi_{\calK}^*\bfPhi_{\calK}^{}$,
meaning \smash{$\set{c_j}_{j\in\calK}$} is a Naimark complement for $\set{\bfphi_j}_{j\in\calK}$.
\end{proof}

We have the following immediate corollary for ETFs $\set{\bfphi_j}_{j=1}^{n}$ with the property that $\ip{\bfphi_j}{\bfphi_{j'}}$ is a root of unity for all $j\neq j'$, such as those considered in~\cite{BodmannPT09,BodmannE10,CoutinkhoGSZ16,FickusJMP18,FickusJMPW19}:

\begin{corollary}
\label{corollary.empty triples}
If $\set{\bfphi_j}_{j=1}^{n}$ is an ETF and $\ip{\bfphi_j}{\bfphi_{j'}}$ is an odd root of unity for all $j\neq j'$, then this ETF contains no regular simplices.
\end{corollary}

For an example of Theorem~\ref{theorem.triple products} in practice,
a direct computation of all \smash{$\binom{10}{3}=120$} triple products of the ETF~\eqref{equation.5x10 ETF} reveals that half of them satisfy~\eqref{equation.triple product condition}, namely:
\begin{align}
\nonumber
&\{1,2,6\},\{1,2,7\},\{1,2,8\},\{1,2,9\},\{1,3,5\},\{1,3,7\},\{1,3,8\},\{1,3,10\},\{1,4,5\},\{1,4,6\},\\
\nonumber
&\{1,4,9\},\{1,4,10\},\{1,5,8\},\{1,5,9\},\{1,6,8\},\{1,6,10\},\{1,7,9\},\{1,7,10\},\{2,3,5\},\{2,3,6\},\\
\label{equation.5x10 triples}
&\{2,3,9\},\{2,3,10\},\{2,4,5\},\{2,4,7\},\{2,4,8\},\{2,4,10\},\{2,5,6\},\{2,5,7\},\{2,6,8\},\{2,7,9\},\\
\nonumber
&\{2,8,10\},\{2,9,10\},\{3,4,6\},\{3,4,7\},\{3,4,8\},\{3,4,9\},\{3,5,6\},\{3,5,8\},\{3,6,7\},\{3,7,10\},\\
\nonumber
&\{3,8,9\},\{3,9,10\},\{4,5,7\},\{4,5,9\},\{4,6,7\},\{4,6,10\},\{4,8,9\},\{4,8,10\},\{5,6,9\},\{5,6,10\},\\
\nonumber
&\{5,7,8\},\{5,7,10\},\{5,8,10\},\{5,9,10\},\{6,7,8\},\{6,7,9\},\{6,8,9\},\{6,9,10\},\{7,8,9\},\{7,8,10\}.
\end{align}
Here, Theorem~\ref{theorem.triple products} says that a $4$-element subset of $\set{1,\dotsc,10}$ corresponds to a regular simplex if and only if every one of its \smash{$\binom43=4$} subsets of cardinality $3$ is one of these triples.
For example, $\set{\bfphi_1,\bfphi_2,\bfphi_6,\bfphi_8}$ is a regular simplex since $\set{1,2,6}$, $\set{1,2,8}$, $\set{1,6,8}$ and $\set{2,6,8}$ all satisfy \eqref{equation.triple product condition}.
Finding all $4$-element subsets of $\set{1,\dotsc,10}$ with this property gives the binder~\eqref{equation.5x10 binder}.

As seen from this example, Theorem~\ref{theorem.triple products} suggests the following iterative meta-algorithm for computing the binder of an ETF $\set{\bfphi_j}_{j=1}^{n}$:
first, exhaustively compute all triples of indices that satisfy~\eqref{equation.triple product condition},
then find all $4$-tuples that consist of $4$ of these triples,
then find all $5$-tuples that consist of $5$ of these $4$-tuples, etc.,
until finally finding all $(s+1)$-tuples that consist of $s+1$ of the $s$-tuples computed in the previous step.
In so doing, we compute those $(s+1)$-tuples with the property that every $3$-element subset of them satisfies~\eqref{equation.triple product condition}, namely the binder.
In Table~\ref{table.binderfinder},
we give a particular instance of this type of algorithm that we call \textit{BinderFinder}.

Essentially, BinderFinder exploits the idea that every $(t+1)$-tuple we want is of the form $\calY=\calX\cup\set{j}$ where $\calX$ is in our list of $t$-tuples, $j\notin\calX$, and $\calX\cup\set{j}\backslash\set{j'}$ is also in our list of $t$-tuples for any $j'\in\calX$.
Such $j$ are characterized by the property that they lie in exactly $t$ of the $t$-tuples in our list that intersect $\calX$ in exactly $t-1$ indices.
For the ETF given in~\eqref{equation.5x10 ETF} for example, we want $4$-tuples consisting of $4$ triples from the list~\eqref{equation.5x10 triples}.
For $4$-tuples that begin with $\set{1,2,6}$ in particular,
we find the other triples that overlap it in two places, namely
$\{1,2,7\}$,
$\{1,2,8\}$,
$\{1,2,9\}$,
$\{1,4,6\}$,
$\{1,6,8\}$,
$\{1,6,10\}$,
$\{2,3,6\}$,
$\{2,5,6\}$ and $\{2,6,8\}$.
Since $8$ is the only index that appears $3$ times in this list and is not in $\set{1,2,6}$,
the only such $4$-tuple is $\set{1,2,6,8}$.
A technical note: by ordering the $t$-tuples and requiring that $i'>i$ in the definition of $\calJ_i$ given in Table~\ref{table.binderfinder}, we ensure the same $(t+1)$-tuple isn't found multiple times.
For example, we don't discover $\set{1,2,6,8}$ a second time when starting with $\set{1,2,8}$.

\begin{table}
\fbox{\parbox{\textwidth}{\begin{itemize}
\item
Via three nested for-loops, compute an initial list of $3$-element subsets of $[n]$:\\
for all $j_1,j_2,j_3=1,\dotsc,n$ with $j_1<j_2<j_3$,
include $\set{j_1,j_2,j_3}$ in this list if~\eqref{equation.triple product condition} holds.
\item
For each $t=3,\dotsc,s$, compute a list of $(t+1)$-element subsets of $[n]$ from a list $\set{\calX_i}_{i=1}^{u_t}$\\
of $t$-element subsets of $[n]$ according to the following process:
for each $i=1,\dotsc,u_t$,
\begin{itemize}
\renewcommand\labelitemii{$\circ$}
\item
find those subsets $\set{\calX_{i'}}_{i'=i+1}^{u_t}$ that overlap with $\calX_i$ in $t-1$ indices, namely
\begin{equation*}
\calJ_i:=\set{i': \#(\calX_i\cap\calX_{i'})=t-1,\ i'>i};
\end{equation*}
\item
find the indices contained in exactly $t$ of those subsets but not in $\calX_i$, namely
\begin{equation*}
\calK_i:=\set{j\in[n]: j\notin\calX_i,\ \sum_{i'\in\calJ_i}\bfone_{\calX_i}(j)=t};
\end{equation*}
\item
for all $j\in\calK_i$, append $\calX_i\cup\set{j}$ to the current list of $(t+1)$-element subsets of $[n]$.
\end{itemize}
\end{itemize}}}
\caption{\label{table.binderfinder}The BinderFinder algorithm for computing the binder of an ETF $\set{\bfphi_j}_{j=1}^{n}$,
namely the set of all subsets of $[n]$ that correspond to a regular simplex that  $\set{\bfphi_j}_{j=1}^{n}$ contains.}
\end{table}

Apart from when $t=3$, BinderFinder never explicitly searches through all $t$-element subsets of $[n]$.
This is important, since \smash{$\binom{n}{t}$} is prohibitively large for all but the smallest values of $n$ and $t$.
In fact, BinderFinder only requires $\rmO(n^3)$ operations when the ETF has no triples that satisfy~\eqref{equation.triple product condition} and thus has an empty binder,
cf.\ Corollary~\ref{corollary.empty triples}.
That said, apart from this and some other cases with empty binders,
BinderFinder is combinatorially complex:
directly computing $\calJ_i$ for all $t$ and $i$ involves \smash{$\sum_{t=3}^{s}\binom{u_t}{2}$} operations,
and if $\set{\bfphi_j}_{j=1}^{n}$ contains even a single regular simplex then \smash{$u_t\geq\binom{s+1}{t}$}.
Here, Vandermonde's identity and Stirling's approximation give
\begin{equation*}
\sum_{t=0}^{s+1}\tbinom{s+1}{t}^2
=\tbinom{2(s+1)}{s+1}
\sim\tfrac{4^{s+1}}{\sqrt{\pi(s+1)}}.
\end{equation*}
Because of this, we in practice use BinderFinder when $n$ and $d$ are small, but still too large to permit a direct search of all $(s+1)$-element subsets of $[n]$.

We also note that even a partial application of BinderFinder can reveal that two ETFs with the same $n$ and $d$ parameters are not equivalent.
To elaborate, if \smash{$\set{j_1,j_2,j_3}$} satisfies~\eqref{equation.triple product condition} for a given ETF,
then~\eqref{equation.triple product invariance} implies that it also satisfies~\eqref{equation.triple product condition} for all ETFs that are projective-unitarily equivalent to it.
Meanwhile, permuting an ETF's indices permutes such a triple accordingly.
As such, if two ETFs are equivalent, then for any $t=3,\dotsc,s+1$,
the list of $t$-tuples produced by BinderFinder for one ETF is related to the corresponding list for the other ETF via a permutation.
In particular,
if two ETFs are equivalent, then they have the same number $u_t$ of such $t$-tuples for all $t=3,\dotsc,s+1$.
For example, if any $10$-vector ETF for $\bbR^5$ does not contain exactly $15$ regular simplices and $60$ $3$-tuples satisfying~\eqref{equation.triple product condition}, then it is not equivalent to~\eqref{equation.5x10 ETF}.

\section{Constructing Naimark complements from binders with combinatorial structure}

We now discuss a surprising application of an ETF's binder:
in some cases, it gives a remarkably simple representation of its Naimark complement.
As we shall see, these ideas give a nontrivial generalization of the recently introduced \textit{phased BIBD ETFs} of~\cite{FickusJMPW19}.
The key idea is to realize that in some cases, the binder of an ETF forms a \textit{balanced incomplete block design} (BIBD).

To be precise, for positive integers $v>k\geq2$ and $\lambda$,
a $\BIBD(v,k,\lambda)$ is a $v$-element vertex set $\calV$ along with a collection $\calB$ of subsets of $\calV$ called \textit{blocks} such that every block contains exactly $k$ vertices, and every pair of distinct vertices is contained in exactly $\lambda$ blocks.
Letting $b$ denote the number of blocks,
this means its $b\times v$ incidence matrix $\bfX$ satisfies $\bfX\bfone=k\bfone$ and that the off-diagonal entries of $\bfX^\rmT\bfX$ are all $\lambda$.
As such, for any $j=1,\dotsc,v$, the number $r_j$ of blocks that contains the $j$th vertex satisfies
\begin{equation*}
(v-1)\lambda
=\sum_{\substack{j'=1\\j'\neq j}}^v(\bfX^\rmT\bfX)(j,j')
=\sum_{i=1}^{b}\bfX(i,j)\sum_{\substack{j'=1\\j'\neq j}}^v\bfX(i,j')
=\sum_{i=1}^{b}\left\{\begin{array}{cl}k-1,&\bfX(i,j)=1\\0,&\bfX(i,j)=0\end{array}\right\}
=r_j(k-1).
\end{equation*}
That is, this number $r_j=r:=\lambda\frac{v-1}{k-1}$ is independent of $j$.
At this point, summing all entries of $\bfX$ gives $bk=vr$ and so \smash{$b=\lambda\frac{v(v-1)}{k(k-1)}$} is also uniquely determined by $v$, $k$ and $\lambda$.

The binder of an ETF $\set{\bfphi_j}_{j=1}^{n}$ for a $d$-dimensional Hilbert space $\bbH$ is by definition a collection of subsets of $\calV=[n]$ of size $s+1$ where $s$ is the inverse Welch bound, namely~\eqref{equation.inverse Welch bound}.
As such, its binder is a $\BIBD(n,s+1,\lambda)$ precisely when any pair of frame vectors is contained in exactly $\lambda$ of the regular simplices it contains.
For example, the binder~\eqref{equation.5x10 binder} of the ETF $\set{\bfphi_j}_{j=1}^{10}$
from~\eqref{equation.5x10 ETF} happens to be a $\BIBD(10,4,2)$ since every pair $\set{\bfphi_j,\bfphi_{j'}}$ of distinct vectors appears in $2$ regular simplices.
Note this implies each single $\bfphi_j$ is contained in \smash{$r=\lambda\frac{v-1}{k-1}=2\frac{9}{3}=6$} regular simplices,
and that there are \smash{$b=\frac{v}{k}r=\frac{10}4 6=15$} regular simplices overall.

As we now explain,
something remarkable happens when the binder of an ETF is a BIBD:
the incidence matrix of this BIBD can be \textit{phased} so as to produce a very sparse Naimark complement of it.
For instance, each one of the $15$ regular simplices $\set{\bfphi_j}_{j\in\calK}$ in~\eqref{equation.5x10 ETF} has a Naimark complement which,
when appropriately scaled,
is a sequence $\set{c_j}_{j\in\calK}$ of unimodular scalars such that $\sum_{j\in\calK}\overline{c_j}\bfphi_j=\bfzero$.
Consider the $15\times 10$ matrix $\bfPsi$ whose every row indicates such a sequence:
\begin{equation}
\label{equation.5x10 phased binder}
\bfPsi=\tfrac1{\sqrt{2}}\left[\begin{array}{rrrrrrrrrr}
1&-1& 0& 0& 0&-1& 0& 1& 0&\phantom{-}0\\
1&-1& 0& 0& 0& 0&-1& 0& 1& 0\\
1& 0&-1& 0&-1& 0& 0& 1& 0& 0\\
1& 0&-1& 0& 0& 0&-1& 0& 0& 1\\
1& 0& 0&-1&-1& 0& 0& 0& 1& 0\\
1& 0& 0&-1& 0&-1& 0& 0& 0& 1\\
0& 1&-1& 0&-1& 1& 0& 0& 0& 0\\
0& 1&-1& 0& 0& 0& 0& 0&-1& 1\\
0& 1& 0&-1&-1& 0& 1& 0& 0& 0\\
0& 1& 0&-1& 0& 0& 0&-1& 0& 1\\
0& 0& 1&-1& 0&-1& 1& 0& 0& 0\\
0& 0& 1&-1& 0& 0& 0&-1& 1& 0\\
0& 0& 0& 0& 1&-1& 0& 0&-1& 1\\
0& 0& 0& 0& 1& 0&-1&-1& 0& 1\\
0& 0& 0& 0& 0& 1&-1&-1& 1& 0
\end{array}\right].
\end{equation}
One can directly verify that $\bfPhi^*\bfPhi+\bfPsi^*\bfPsi=12\bfI$,
meaning the columns \smash{$\set{\bfpsi_j}_{j=1}^{10}$} of $\bfPsi$ form a Naimark complement for \smash{$\set{\bfphi_j}_{j=1}^{10}$}.
As such, $\set{\bfpsi_j}_{j=1}^{10}$ is an ETF for its span which has dimension $n-d=10-5=5$.
We now prove that this holds in general.
Here, it turns out that we only need a subset of the binder to be a BIBD:

\begin{theorem}
\label{theorem.phased binder}
Let $\set{\bfphi_j}_{j=1}^{n}$ be an ETF for a $d$-dimensional Hilbert space,
and without loss of generality assume \smash{$\norm{\bfphi_j}^2=s$} for all $j$ where $s$ is the inverse Welch bound~\eqref{equation.inverse Welch bound}.
For any $i=1,\dotsc,b$, let $\set{\bfphi_j}_{j\in\calK_i}$ be a regular simplex and take scalars $\set{c_{i,j}}_{j\in\calK_i}$ such that~\eqref{equation.regular simplex Naimark complement} holds.
Then these scalars satisfy
\begin{equation}
\label{equation.consistency of relative phase}
\overline{c_{i,j}}c_{i,j'}=\overline{c_{i',j}}c_{i',j'},\quad
\forall\,i\neq i',\, j\neq j'
\text{ such that }
j,j'\in\calK_i\cap\calK_{i'}.
\end{equation}
Moreover, if $\set{\calK_i}_{i=1}^b$ is a $\BIBD(n,s+1,\lambda)$ for some $\lambda>0$,
then the columns $\set{\bfpsi_j}_{j=1}^{n}$ of
\begin{equation}
\label{equation.phased incidence matrix}
\bfPsi\in\bbF^{b\times n},
\quad
\bfPsi(i,j):=\tfrac1{\sqrt{\lambda}}\left\{\begin{array}{cl}c_{i,j},&j\in\calK_i,\\0,&j\notin\calK_i,\end{array}\right.
\end{equation}
form a Naimark complement of $\set{\bfphi_j}_{j=1}^{n}$.
\end{theorem}

\begin{proof}
For any $i=1,\dotsc,b$,
the fact that $\set{c_{i,j}}_{j\in\calK_i}$ satisfies~\eqref{equation.regular simplex Naimark complement} means that it is a Naimark complement for $\set{\bfphi_j}_{j\in\calK_i}$,
and so $\overline{c_{i,j}}c_{i,j'}=-\ip{\bfphi_j}{\bfphi_{j'}}$ for all $j,j'\in\calK_i$ such that $j\neq j'$.
In particular, \eqref{equation.consistency of relative phase} holds.

Continuing, if $\set{\calK_i}_{i=1}^b$ is a $\BIBD(n,s+1,\lambda)$ then for any $j\neq j'$, there are exactly $\lambda$ values of $i$ such that both $\bfPsi(i,j)$ and $\bfPsi(i,j')$ are nonzero, namely those $i$ such that $j,j'\in\calK_i$.
As such,
\begin{equation*}
(\bfPhi^*\bfPhi)(j,j')
=\sum_{i=1}^{b}\overline{\bfPhi(i,j)}\bfPhi(i,j')
=\sum_{\set{i: j,j'\in\calK_i}}\tfrac1{\lambda}\,\overline{c_{i,j}}c_{i,j'}
=-\tfrac1{\lambda}\sum_{\set{i: j,j'\in\calK_i}}\ip{\bfphi_j}{\bfphi_{j'}}
=-\ip{\bfphi_j}{\bfphi_{j'}}
\end{equation*}
for any $j\neq j'$.
Moreover, if \smash{$\set{\calK_i}_{i=1}^b$} is a $\BIBD(n,s+1,\lambda)$, each column of $\bfPsi$ has exactly \smash{$r=\lambda\frac{v-1}{k-1}=\lambda\frac{n-1}{s}$} nonzero entries,
each of modulus \smash{$\lambda^{-\frac12}$},
implying $(\bfPsi^*\bfPsi)(j,j)=\frac{n-1}{s}$ for all $j=1,\dotsc,n$.
In particular, since \smash{$s^2=\tfrac{d(n-1)}{n-d}$}, every diagonal entry of $\bfPhi^*\bfPhi+\bfPsi^*\bfPsi$ has value
\begin{equation*}
s+\tfrac{n-1}s
=s(1+\tfrac{n-1}{s^2})
=s(1+\tfrac{n-d}{d})
=\tfrac{ns}d,
\end{equation*}
which happens to equal the tight frame constant $a$ of $\set{\bfphi_j}_{j=1}^{n}$.
As such, $\bfPhi^*\bfPhi+\bfPsi^*\bfPsi=a\bfI$,
meaning the columns of $\bfPsi$ form a Naimark complement for $\set{\bfphi_j}_{j=1}^{n}$.
\end{proof}

Essentially what this result says is that when the binder of an ETF---or a subset of it---forms a BIBD, we can piece together Naimark complements of its regular simplices in order to form a Naimark complement for the entire ETF.

As we now explain, the ETFs that arise from Theorem~\ref{theorem.phased binder} are a generalization of the \textit{phased BIBD ETFs} of~\cite{FickusJMPW19}.
There, it is observed that if $\bfX$ is the $\set{0,1}$-valued $b\times n$ incidence matrix of a $\BIBD(n,k,1)$,
and if $\bfPsi$ is any matrix obtained by \textit{phasing} $\bfX$,
that is, multiplying its entries by unimodular scalars,
then the columns $\set{\bfpsi_j}_{j=1}^{n}$ of $\bfPsi$ are automatically equiangular,
since for any $j\neq j'$, $\ip{\bfpsi_j}{\bfpsi_{j'}}$ is simply the product of two unimodular numbers.
The challenge then becomes to find a way to phase $\bfX$ so that $\set{\bfpsi_j}_{j=1}^{n}$ is an ETF for its span, namely so that $\bfPsi\bfPsi^*\bfPsi=a\bfPsi$ for some $a>0$.
By Theorem~3.4 of~\cite{FickusJMPW19}, this span is necessarily of dimension
\begin{equation}
\label{equation.dimension of phased BIBD ETF}
\tfrac{n(n-1)}{(n-1)+(k-1)^2}.
\end{equation}
Two nontrivial infinite families of such ETFs are constructed in~\cite{FickusJMPW19}:
for any prime power $q$, one can phase a $\BIBD(q^2,q,1)$ and a $\BIBD(q^3+1,q+1,1)$ to produce ETFs for spaces of dimension \smash{$\binom{q+1}{2}$} and $q(q^2-q+1)$, respectively.
For example, when $q=3$, the columns $\set{\bfpsi_j}_{j=1}^{9}$ of
\newlength{\NewLengthOne}
\settowidth{\NewLengthOne}{$z^2$}
\begin{equation}
\label{equation.6x9 phased BIBD ETF}
\setlength{\arraycolsep}{2pt}
\bfPsi=\left[\begin{array}{lllllllll}
\makebox[\NewLengthOne][l]{1}&\makebox[\NewLengthOne][l]{1}&\makebox[\NewLengthOne][l]{1}&  0&  0&  0&  0&  0&  0\\
  0&  0&  0&  1&  1&  1&  0&  0&  0\\
  0&  0&  0&  0&  0&  0&  1&  1&  1\\
  1&  0&  0&  1&  0&  0&  1&  0&  0\\
  0&  1&  0&  0&z^2&  0&  0&  z&  0\\
  0&  0&  1&  0&  0&  z&  0&  0&z^2\\
  1&  0&  0&  0&  0&z^2&  0&z^2&  0\\
  0&  1&  0&  z&  0&  0&  0&  0&  1\\
  0&  0&  1&  0&  1&  0&  z&  0&  0\\
  1&  0&  0&  0&  z&  0&  0&  0&  z\\
  0&  1&  0&  0&  0&  1&z^2&  0&  0\\
  0&  0&  1&z^2&  0&  0&  0&  1&  0\\
\end{array}\right],
\quad
z=\exp(\tfrac{2\pi\rmi}3),
\end{equation}
form a phased $\BIBD(9,3,1)$ ETF for their $6$-dimensional span.

In the context of~\cite{FickusJMPW19}, the interesting thing about Theorem~\ref{theorem.phased binder} is that it generalizes that theory so as to permit BIBDs with $\lambda>1$, such as~\eqref{equation.5x10 phased binder}.
This is a nontrivial generalization,
since if $\bfPsi$ is obtained by phasing the incidence matrix of a $\BIBD(v,k,\lambda)$ where $\lambda>1$ then its columns are not automatically equiangular.
Instead, this only holds under certain hypotheses,
such as~\eqref{equation.consistency of relative phase}.
In the next result, we verify this fact and then prove a type of converse to Theorem~\ref{theorem.phased binder}.

\begin{theorem}
\label{theorem.Naimark complement of phased BIBD ETF}
Let $\set{\calK_i}_{i=1}^b$ be subsets of $[n]$ that form a $\BIBD(n,k,\lambda)$ for some $k\geq 2$, $\lambda>0$.
For each $i=1,\dotsc,b$, let $\set{c_{i,j}}_{j\in\calK_i}$ be unimodular scalars that satisfy
\eqref{equation.consistency of relative phase}.
Then the columns $\set{\bfpsi_j}_{j=1}^{n}$ of the matrix $\bfPsi$ defined by~\eqref{equation.phased incidence matrix} are equiangular.
Moreover, if $\set{\bfpsi_j}_{j=1}^{n}$ is an ETF for its span,
then this span is necessarily of dimension~\eqref{equation.dimension of phased BIBD ETF},
and each subset $\set{\bfphi_j}_{j\in\calK_i}$ of any Naimark complement $\set{\bfphi_j}_{j=1}^{n}$ of $\set{\bfpsi_j}_{j=1}^{n}$ is a regular simplex.
\end{theorem}

\begin{proof}
Since $\set{\calK_i}_{i=1}^b$ is a $\BIBD(n,k,\lambda)$,
each column $\bfpsi_j$ of $\bfPsi$ contains exactly $r=\lambda\frac{n-1}{k-1}$ nonzero entries, each of modulus \smash{$\lambda^{-\frac12}$}, and so \smash{$\norm{\bfpsi_j}^2=\frac{n-1}{k-1}$}.
Moreover, for any $j\neq j'$,
\eqref{equation.consistency of relative phase} and~\eqref{equation.phased incidence matrix} imply
\begin{equation*}
\ip{\bfpsi_j}{\bfpsi_{j'}}
=(\bfPsi^*\bfPsi)(j,j')
=\sum_{i=1}^b\overline{\bfPsi(i,j)}\bfPsi(i,j')
=\sum_{\set{i: j,j'\in\calK_i}}\tfrac1{\lambda}\,\overline{c_{i,j}}c_{i,j'}
=\overline{c_{i',j}}c_{i',j'},
\end{equation*}
where $i'$ is any fixed index such that $j,j'\in\calK_i$,
meaning $\abs{\ip{\bfpsi_j}{\bfpsi_{j'}}}=\abs{c_{i',j}}\abs{c_{i',j'}}=1$.
Thus, having~\eqref{equation.consistency of relative phase} indeed implies that $\set{\bfpsi_j}_{j=1}^{n}$ is equiangular.

Now further assume that $\set{\bfpsi_j}_{j=1}^{n}$ is an ETF for its span, and let $\set{\bfphi_j}_{j=1}^{n}$ be any Naimark complement of it.
Letting $d=\dim(\Span\set{\bfphi_j}_{j=1}^{n})$, we have that $\set{\bfpsi_j}_{j=1}^{n}$ is an ETF for a space of dimension $n-d$, and so its coherence is:
\begin{equation}
\label{equation.proof of Naimark complement of phased BIBD ETF 1}
\bigbracket{\tfrac{d}{(n-d)(n-1)}}^{\frac12}
=\bigbracket{\tfrac{n-(n-d)}{(n-d)(n-1)}}^{\frac12}
=\max_{j\neq j'}\tfrac{\abs{\ip{\bfpsi_j}{\bfpsi_{j'}}}}{\norm{\bfpsi_j}\norm{\bfpsi_{j'}}}
=\max_{j\neq j'}\tfrac{k-1}{n-1}
=\tfrac{k-1}{n-1}.
\end{equation}
Manipulating~\eqref{equation.proof of Naimark complement of phased BIBD ETF 1} gives an expression for $n-d=\dim(\Span\set{\bfpsi_j}_{j=1}^{n})$ in terms of $n$ and $k$, namely~\eqref{equation.dimension of phased BIBD ETF}.
Moreover,~\eqref{equation.proof of Naimark complement of phased BIBD ETF 1} implies that the inverse Welch bound~\eqref{equation.inverse Welch bound} for $\set{\bfphi_j}_{j=1}^{n}$ is
\begin{equation*}
s
=\bigbracket{\tfrac{d(n-1)}{n-d}}^{\frac12}
=(n-1)\bigbracket{\tfrac{d}{(n-d)(n-1)}}^{\frac12}
=(n-1)\tfrac{k-1}{n-1}
=k-1.
\end{equation*}
As such, Theorem~\ref{theorem.spark bound equality} gives that a subset $\set{\bfphi_j}_{j\in\calK}$ of $\set{\bfphi_j}_{j=1}^{n}$ is a regular simplex if it consists of $s+1=k$ linearly dependent vectors.
This holds whenever $\calK=\calK_i$ for some $i=1,\dotsc,b$:
each $\calK_i$ has cardinality $k$ by assumption,
and applying~\eqref{equation.cross frame operator of Naimark complements} to the $i$th standard basis element $\bfdelta_i$ for $\bbF^b$ gives
\begin{equation*}
\bfzero
=\bfzero\bfdelta_i
=\bfPhi\bfPsi^*\bfdelta_i
=\sum_{j=1}^n(\bfPsi^*\bfdelta_i)(j)\bfphi_j
=\sum_{j=1}^n\overline{\bfPsi(i,j)}\bfphi_j
=\sum_{j\in\calK_i}\overline{c_{i,j}}\bfphi_j.\qedhere
\end{equation*}
\end{proof}

Applying this result to \eqref{equation.6x9 phased BIBD ETF} for example gives that any one of its Naimark complements $\set{\bfphi_j}_{j=1}^{9}$ is an ETF for $\bbC^3$ that contains at least $12$ simplices, one indicated by each row of $\bfPsi$.
More generally, Theorem~\ref{theorem.Naimark complement of phased BIBD ETF} implies that a Naimark complement of the phased $\BIBD(q^2,q,1)$ ETF of~\cite{FickusJMPW19} is a $q^2$-vector ETF for a space of dimension \smash{$\binom{q}{2}$} whose binder contains that BIBD, and so contains at least $q(q+1)$ regular simplices.
Similarly, this result implies that a Naimark complement of the phased $\BIBD(q^3+1,q+1,1)$ ETF of~\cite{FickusJMPW19} is a $(q^3+1)$-vector ETF for a space of dimension \smash{$\frac{q^3+1}{q+1}$} whose binder contains that BIBD, and so contains at least \smash{$q^2(\frac{q^3+1}{q+1})$} regular simplices.

Sometimes, the Naimark complement of a phased BIBD ETF contains more simplices than those indicated by the original BIBD.
For example, consider the phased $\BIBD(28,4,1)$ ETF $\set{\bfpsi_j}_{j=1}^{28}$ given in Figure~1 of~\cite{FickusJMPW19},
which is a $28$-vector ETF for $\bbF^{21}$ where $\bbF$ is either $\bbR$ or $\bbC$ depending on whether the parameter $z$ is chosen to be $-1$ or $\pm\rmi$, respectively.
When $z=\pm\rmi$, it turns out that applying BinderFinder to any one of its Naimark complements $\set{\bfphi_j}_{j=1}^{28}$ reveals that $\set{\bfphi_j}_{j=1}^{28}$ contains exactly $63$ simplices, namely those indicated by the $63$ blocks of the original $\BIBD(28,4,1)$.
In contrast, there are five times as many simplices in the real case:
when $z=-1$, BinderFinder reveals that the binder of $\set{\bfphi_j}_{j=1}^{28}$ is a $\BIBD(28,4,5)$.
By Theorem~\ref{theorem.Naimark complement of phased BIBD ETF},
the $\BIBD(28,4,5)$ that arises as the binder of $\set{\bfphi_j}_{j=1}^{28}$
necessarily contains the original $\BIBD(28,4,1)$.
Moreover, by Theorem~\ref{theorem.phased binder}, the Gram matrix $\bfPsi^*\bfPsi$ of the original phased $\BIBD(28,4,1)$ can be obtained by phasing this $\BIBD(28,4,5)$ according to the linear dependence relations~\eqref{equation.regular simplex Naimark complement} on the corresponding simplices in $\set{\bfphi_j}_{j=1}^{28}$.

In particular, we see that the $28$-vector ETF for $\bbR^{21}$ constructed in~\cite{FickusJMPW19} has multiple distinct representations as a phased BIBD ETF: it can be obtained by phasing the incidence matrix of a $\BIBD(28,4,1)$ or of a larger $\BIBD(28,4,5)$, or indeed of a $\BIBD(28,4,4)$ that consists of those blocks in this $\BIBD(28,4,5)$ that are not blocks in this $\BIBD(28,4,1)$.
More trivially, we can always vertically concatenate $\kappa$ copies of the synthesis operator of a phased $\BIBD(n,s+1,\lambda)$ ETF to form a phased $\BIBD(n,s+1,\kappa\lambda)$ representation of it.
This also explains why the dimension~\eqref{equation.dimension of phased BIBD ETF} of the span of a phased BIBD ETF is not dependent on $\lambda$.

These observations suggest the following problem: when the binder of an ETF contains a $\BIBD(n,s+1,\lambda)$, what is the smallest $\lambda$ for which this holds?
To be clear, one cannot always take $\lambda=1$:
as noted above,
the binder~\eqref{equation.5x10 binder} of the ETF~\eqref{equation.5x10 ETF} is a $\BIBD(10,4,2)$,
and a $\BIBD(10,4,1)$ does not exist since the number of blocks such a design would necessarily contain is \smash{$\lambda\frac{v(v-1)}{k(k-1)}=7.5$}, which is not an integer.
In this sense,~\eqref{equation.5x10 phased binder} gives the smallest possible phased BIBD representation of a Naimark complement of~\eqref{equation.5x10 ETF}.
As such, solving this problem requires the generalization of~\cite{FickusJMPW19} that we have provided here.
We leave a deeper investigation of this problem for future research.

To be clear, the material presented in this section only applies when the binder of our ETF contains a BIBD.
This is not always the case.
Indeed, sometimes the binder is empty.
For example, the off-diagonal entries of the Gram matrix of~\eqref{equation.6x9 phased BIBD ETF} are cube roots of unity, and so Corollary~\ref{corollary.empty triples} implies that it contains no regular simplices despite the fact that its columns form a $9$-vector ETF for a $6$-dimensional space where the inverse Welch bound is the integer \smash{$\bracket{\tfrac{6(9-1)}{9-6}}^{\frac12}=4$}.
(In contrast, its Naimark complements contain 12 simplices, one for each row of~\eqref{equation.6x9 phased BIBD ETF}.)
Moreover, as we shall see in the coming sections, there are other examples of ETFs whose binders are not empty and yet do not contain any BIBD.
To date, we have not been able to find any general conditions which guarantee that an ETF's binder contains a BIBD.
That said, using techniques similar to those of~\cite{BodmannE10,BodmannE10b}, we have been able to show that under certain conditions,
the set of $3$-element subsets of $[n]$ that satisfy~\eqref{equation.triple product condition} does indeed form a $\BIBD(n,3,\lambda)$ for some $\lambda>0$:

\begin{theorem}
\label{theorem.BIBD of triples}
Let $p$ be a prime,
and let $\set{\bfphi_j}_{j=1}^{n}$ be an equiangular tight frame with the property that  $-\ip{\bfphi_j}{\bfphi_{j'}}$ is a $p$th root of unity for any $j\neq j'$.
Then the set of $3$-element subsets of $[n]$ that satisfy~\eqref{equation.triple product condition} is a $\BIBD(n,3,\lambda)$ with
\begin{equation}
\label{equation.BIBD of triples lambda}
\lambda=\tfrac{n-2}{p}-\tfrac{p-1}{p}(\tfrac{n}{d}-2)\bigbracket{\tfrac{d(n-1)}{n-d}}^{\frac12}.
\end{equation}
\end{theorem}

\begin{proof}
Note that in order for $\ip{\bfphi_j}{\bfphi_{j'}}$ to be the negatives of $p$th roots of unity,
we are implicitly requiring that these inner products are unimodular, namely that $\norm{\bfphi_j}^2=s$ for all $j$, where $s$ is the inverse Welch bound~\eqref{equation.inverse Welch bound}.  As such, this ETF has tight frame constant $a=\frac{ns}{d}$ and Lemma~\ref{lemma:tight frame for its span} gives
$\tfrac{ns}{d}\bfPhi^*\bfPhi=(\bfPhi^*\bfPhi)^2$.
For any $j\neq j'$, taking the $(j,j')$th entries of this equation gives
\begin{equation*}
\tfrac{ns}{d}\ip{\bfphi_j}{\bfphi_{j'}}
=\sum_{j''=1}^n\ip{\bfphi_j}{\bfphi_{j''}}\ip{\bfphi_{j''}}{\bfphi_{j'}}
=2s\ip{\bfphi_j}{\bfphi_{j'}}
+\sum_{\substack{j''=1\\j''\neq j,j'}}^n\ip{\bfphi_j}{\bfphi_{j''}}\ip{\bfphi_{j''}}{\bfphi_{j'}}.
\end{equation*}
For each $j\neq j'$, let $\theta(j,j')$ be an integer such that $\ip{\bfphi_j}{\bfphi_{j'}}=-z^{\theta(j,j')}$ where \smash{$z=\exp(\frac{2\pi\rmi}{p})$}.
Under this notation, the previous equation is equivalent to having
\begin{equation}
\label{equation.proof of BIBD of triples 1}
0
=(\tfrac{n}{d}-2)s+\sum_{\substack{j''=1\\j''\neq j,j'}}^n z^{\theta(j,j'')+\theta(j'',j')-\theta(j,j')}.
\end{equation}

Now note that since $p$ is prime, the primitive $p$th roots of unity $\set{z^i}_{i=1}^{p}$ are linearly independent over $\bbQ$.
As such the operator $\bfZ:\bbQ^p\rightarrow\bbC$,
\smash{$\bfZ\set{c_i}_{i=0}^{p-1}:=\sum_{i=0}^{p-1}c_i z^i$} has rank at least $p$, meaning the null space of $\bfZ$ is at most one-dimensional.
When combined with the fact that $\sum_{i=0}^{p-1}z^i=0$, this implies that the null space of $\bfZ$ is exactly one-dimensional, and consists of all constant vectors.
That is, if rational scalars $\set{c_i}_{i=0}^{p-1}$ satisfy $\sum_{i=0}^{p-1}c_i z^i=0$ then they are necessarily all equal.

To exploit this fact here, for any $i=0,\dotsc,p-1$, we let
\begin{equation}
\label{equation.proof of BIBD of triples 2}
m_{j,j',i}
:=\#\set{j''=1,\dotsc,n: j''\neq j,j',\ \theta(j,j'')+\theta(j'',j')-\theta(j,j')\equiv i\bmod p}
\end{equation}
and rewrite \eqref{equation.proof of BIBD of triples 1} in terms of these nonnegative integers:
\begin{equation*}
0
=(\tfrac{n}{d}-2)s+\sum_{i=0}^{p-1}m_{j,j',i}z^i
=[(\tfrac{n}{d}-2)s+m_{j,j',0}]z^0+\sum_{i=1}^{p-1}m_{j,j',i}z^i.
\end{equation*}
As such, $m_{j,j',i}=(\tfrac{n}{d}-2)s+m_{j,j',0}$ for all $i=1,\dotsc,p-1$.
Moreover,
the values $\set{m_{j,j',i}}_{i=0}^{p-1}$ are the cardinalities of sets that partition $\set{j''=1,\dotsc,n: j''\neq j,j'}$, meaning they sum to $n-2$.
Together, these facts imply
\begin{equation}
\label{equation.proof of BIBD of triples 3}
n-2
=\sum_{i=0}^{p-1}m_{j,j',i}
=m_{j,j',0}+(p-1)[(\tfrac{n}{d}-2)s+m_{j,j',0}]
=p\,m_{j,j',0}+(p-1)(\tfrac{n}{d}-2)s.
\end{equation}

Now consider the set of all $3$-element subsets $\set{j_1,j_2,j_3}$ of $[n]$ such that~\eqref{equation.triple product condition} holds.
Our claim is that this set is a $\BIBD(n,3,\lambda)$ where $\lambda$ is given by~\eqref{equation.BIBD of triples lambda}.
That is, for any $j\neq j'$, we want to show there are exactly $\lambda$ elements in the set
\begin{align*}
&\set{j''=1,\dotsc,n: j''\neq j,j',\ \ip{\bfphi_j}{\bfphi_{j'}}
\ip{\bfphi_{j'}}{\bfphi_{j''}}\ip{\bfphi_{j''}}{\bfphi_{j}}=-1}\\
&\quad=\set{j''=1,\dotsc,n: j''\neq j,j',\ z^{\theta(j,j')+\theta(j',j'')+\theta(j'',j)}=1}\\
&\quad=\set{j''=1,\dotsc,n: j''\neq j,j',\ \theta(j,j')-\theta(j'',j')-\theta(j,j'')\equiv 0\bmod p}.
\end{align*}
Comparing this to~\eqref{equation.proof of BIBD of triples 2},
this reduces to showing that $m_{j,j',0}=\lambda$ for all $j\neq j'$,
which immediately follows from solving for $m_{j,j',0}$ in~\eqref{equation.proof of BIBD of triples 3}.
\end{proof}

In the particular case where the ETF is real, the hypothesis of Theorem~\ref{theorem.BIBD of triples} is automatically satisfied with $p=2$, provided $\set{\bfphi_j}_{j=1}^{n}$ is without loss of generality scaled so that $\norm{\bfphi_j}^2=s$ for all $j$ where $s$ is the inverse Welch bound~\eqref{equation.inverse Welch bound}.
In this case, \eqref{equation.BIBD of triples lambda} reduces to \smash{$\lambda=\tfrac{n-2}2-\tfrac12(\tfrac{n}{d}-2)\bracket{\tfrac{d(n-1)}{n-d}}^{\frac12}$,}
meaning the number of triples that satisfy~\eqref{equation.triple product condition} is exactly
\begin{equation*}
\lambda\tfrac{n(n-1)}{3(3-1)}
=\tfrac12\tbinom{n}{3}-\tfrac1{12}n(n-1)(\tfrac{n}{d}-2)\bigbracket{\tfrac{d(n-1)}{n-d}}^{\frac12}.
\end{equation*}
When $n=2d$, this means that exactly half of all possible $3$-element subsets of a real ETF $\set{\bfphi_j}_{j=1}^{n}$ satisfy~\eqref{equation.triple product condition} as evidenced, for example, by a list~\eqref{equation.5x10 triples} of such subsets that arises in the $d=5$ case.

\section{Equiangular tight frames that are a disjoint union of regular simplices}

We say an ETF $\set{\bfphi_j}_{j=1}^{n}$ for a space of dimension $d$ is a \textit{disjoint union of regular simplices} if its indices can be partitioned into sets $\set{\calK_i}_{i=1}^{v}$ such that $\set{\bfphi_j}_{j\in\calK_i}$ is a regular simplex for any $i=1,\dotsc,v$.
Since each $\calK_i$ necessarily contains $s+1$ elements where $s$ is the inverse Welch bound~\eqref{equation.inverse Welch bound}, this requires that $n=v(s+1)$.
In particular, $s+1$ must divide $n$.
In this section, we discuss two results regarding such ETFs, namely that they have a Naimark complement of a particular form, and that they lead to certain optimal packings of subspaces called \textit{equichordal tight fusion frames}.

Steiner ETFs are unions of simplices by design~\cite{FickusMT12}.
To elaborate, let $\bfX$ be the $b\times v$ incidence matrix of a $\BIBD(v,k,1)$,
which is also known as a \textit{$(2,k,v)$-Steiner system}.
In this $\lambda=1$ setting, each vertex in this block design is contained in exactly \smash{$r=\frac{v-1}{k-1}$} blocks, and there are \smash{$b=\frac{v}{k}r=\frac{v(v-1)}{k(k-1)}$} blocks total.
For every $i=1,\dotsc,v$,
we form a corresponding $b\times r$ \textit{embedding matrix} $\bfE_i$ whose columns are standard basis elements that sum to the $i$th column $\bfx_i$ of $\bfX$.
For example, for a $\BIBD(4,2,1)$ we can take:
\begin{equation}
\label{equation.6x4 incidence matrix and embedding operators}
\bfX
=\left[\begin{array}{cccc}
1&1&0&0\\
0&0&1&1\\
1&0&1&0\\
0&1&0&1\\
1&0&0&1\\
0&1&1&0
\end{array}\right],\quad
\bfE_1
=\left[\begin{array}{ccc}
1&0&0\\
0&0&0\\
0&1&0\\
0&0&0\\
0&0&1\\
0&0&0
\end{array}\right],\quad
\bfE_2
=\left[\begin{array}{ccc}
1&0&0\\
0&0&0\\
0&0&0\\
0&1&0\\
0&0&0\\
0&0&1
\end{array}\right],\quad
\bfE_3
=\left[\begin{array}{ccc}
0&0&0\\
1&0&0\\
0&1&0\\
0&0&0\\
0&0&0\\
0&0&1
\end{array}\right],\quad
\bfE_4
=\left[\begin{array}{ccc}
0&0&0\\
1&0&0\\
0&0&0\\
0&1&0\\
0&0&1\\
0&0&0
\end{array}\right].
\end{equation}
We then use $\set{\bfE_i}_{i=1}^{v}$ to isometrically embed $v$ copies of a (unimodular) \textit{flat} regular simplex for $\bbF^r$,
namely the columns $\set{\bff_j}_{j=1}^{r+1}$ of an $r\times(r+1)$ matrix $\bfF$ obtained by removing a single row from a possibly-complex Hadamard matrix.
Doing so produces a sequence of $n=v(r+1)$ vectors \smash{$\set{\bfE_i\bff_j}_{i=1,}^{v}\,_{j=1}^{r+1}$} which,
as we explain below, is an ETF for $\bbF^d$ where $d=b$.
For example, when $r=3$ we can remove the first row of the canonical $4\times 4$ Hadamard matrix to obtain
\begin{equation*}
\bfF
=\left[\begin{array}{cccc}\bff_1&\bff_2&\bff_3&\bff_4\end{array}\right]
=\left[\begin{array}{rrrr}
 1&-1& 1&-1\\
 1& 1&-1&-1\\
 1&-1&-1& 1
\end{array}\right],
\end{equation*}
and applying $\set{\bfE_i}_{i=1}^{4}$ to $\set{\bff_j}_{j=1}^{4}$ gives
\smash{$\bfPhi=\left[\begin{array}{cccc}
\bfE_1\bfF&\bfE_2\bfF&\bfE_3\bfF&\bfE_4\bfF
\end{array}\right]$},
that is,
\begin{equation}
\label{equation.6x16 Steiner ETF}
\bfPhi
=\left[\begin{array}{rrrrrrrrrrrrrrrr}
 1&-1& 1&-1&\phantom{-}1&-1& 1&-1&\phantom{-}0& 0& 0& 0&\phantom{-}0& 0& 0& 0\\
 0& 0& 0& 0& 0& 0& 0& 0& 1&-1& 1&-1& 1&-1& 1&-1\\
 1& 1&-1&-1& 0& 0& 0& 0& 1& 1&-1&-1& 0& 0& 0& 0\\
 0& 0& 0& 0& 1& 1&-1&-1& 0& 0& 0& 0& 1& 1&-1&-1\\
 1&-1&-1& 1& 0& 0& 0& 0& 0& 0& 0& 0& 1&-1&-1& 1\\
 0& 0& 0& 0& 1&-1&-1& 1& 1&-1&-1& 1& 0& 0& 0& 0
\end{array}\right].
\end{equation}
For each $i$, $\bfE_i$ is an isometry, that is, $\bfE_i^*\bfE_i^{}=\bfI$.
As such, each subset $\set{\bfE_i\bff_j}_{j=1}^{r+1}$ is a regular simplex with $\norm{\bfE_i\bff_j}^2=\norm{\bff_j}^2=r$ for all $j$ and $\abs{\ip{\bfE_i\bff_j}{\bfE_i\bff_{j'}}}=\abs{\ip{\bff_j}{\bff_{j'}}}=1$ for all $j\neq j'$.
Meanwhile, for any $i\neq i'$,
the fact that $\bfX$ is the incidence matrix of a $\BIBD(v,k,1)$ implies its $i$th and $i'$th columns only have one index of common support,
implying $\bfE_i^*\bfE_{i'}^{}$ has exactly one nonzero entry, and this entry has value $1$.
As such, $\bfE_i^*\bfE_{i'}^{}=\bfdelta_{l}^{}\bfdelta_{l'}^*$ for some standard basis elements $\bfdelta_{l},\bfdelta_{l'}$ for $\bbF^s$ that depend on $i$ and $i'$.
As such, for any $i,i'$ and $j,j'$ with $i\neq i'$,
\begin{equation*}
\ip{\bfE_i\bff_j}{\bfE_{i'}\bff_{j'}}
=\ip{\bff_j}{\bfE_i^*\bfE_{i'}^{}\bff_{j'}}
=\ip{\bff_j}{\bfdelta_{l}^{}\bfdelta_{l'}^*\bff_{j'}}
=\ip{\bfdelta_{l}^*\bff_j}{\bfdelta_{l'}^*\bff_{j'}}
=\overline{\bff_j(l)}\bff_{j'}(l'),
\end{equation*}
meaning such inner products are also unimodular, being the product of two unimodular numbers.

Together, these facts imply that
\smash{$\set{\bfE_i\bff_j}_{i=1,}^{v}\,_{j=1}^{r+1}$} is a disjoint union of $v$ regular $r$-simplices, and that it is equiangular with coherence $\frac1r$.
Moreover, since $vk=br$ and $v-1=r(k-1)$, this sequence of vectors has redundancy
\begin{equation}
\label{equation.redundancy of Steiner ETF}
\tfrac nd
=\tfrac{v}{b}(r+1)
=\tfrac{k}{r}(r+1)
=k(1+\tfrac1r)
>k\geq 2.
\end{equation}
In particular, dividing $\frac{n-d}{d}=k(1+\tfrac1r)-1=\tfrac1r(v+k-1)$ by $n-1=vr+(v-1)=r(v+k-1)$ gives a Welch bound of $\frac1r$.
Since the coherence of \smash{$\set{\bfE_i\bff_j}_{i=1,}^{v}\,_{j=1}^{s+1}$} achieves this bound, these vectors are necessarily an ETF for $\bbF^b$,
with an inverse Welch bound~\eqref{equation.inverse Welch bound} of $s=r$.
(Alternatively, it is straightforward to show the corresponding synthesis operator $\bfPhi$ has equal-norm orthogonal rows~\cite{FickusMT12}.)
This type of construction was introduced in~\cite{GoethalsS70} as a method for obtaining SRGs.
In~\cite{FickusMT12} it was rediscovered and recognized as a method for constructing ETFs.
These ETFs are real if the flat regular simplex $\set{\bff_j}_{j=1}^{r+1}$ is real, namely when it is obtained by removing a row from a real Hadamard matrix of size $r+1$, which requires $r+1$ to be divisible by $4$~\cite{FickusJMP19}.
Otherwise, one can always obtain a complex flat regular simplex $\set{\bff_j}_{j=1}^{r+1}$ by removing a row from a discrete Fourier transform matrix of size $r+1$, for example.

It turns out that Steiner ETFs are not the only ETFs that are a disjoint union of simplices.
For example, as detailed and generalized in the next section,
$\set{1,2,3,4,8,11,12,14}$ is an $8$-element difference set in $\bbZ_{15}$,
and so extracting the corresponding $8$ rows of the (inverse) discrete Fourier transform (DFT) of size $15$ yields a matrix whose columns $\set{\bfphi_j}_{j=0}^{14}$ form an ETF for $\bbC^8$:
\begin{equation}
\label{equation.8x15 ETF}
\bfPhi=\tfrac1{\sqrt{2}}\left[\begin{array}{lllllllllllllll}
z^{0}&z^{1}&z^{2}&z^{3}&z^{4}&z^{5}&z^{6}&z^{7}&z^{8}&z^{9}&z^{10}&z^{11}&z^{12}&z^{13}&z^{14}\\
z^{0}&z^{2}&z^{4}&z^{6}&z^{8}&z^{10}&z^{12}&z^{14}&z^{1}&z^{3}&z^{5}&z^{7}&z^{9}&z^{11}&z^{13}\\
z^{0}&z^{3}&z^{6}&z^{9}&z^{12}&z^{0}&z^{3}&z^{6}&z^{9}&z^{12}&z^{0}&z^{3}&z^{6}&z^{9}&z^{12}\\
z^{0}&z^{4}&z^{8}&z^{12}&z^{1}&z^{5}&z^{9}&z^{13}&z^{2}&z^{6}&z^{10}&z^{14}&z^{3}&z^{7}&z^{11}\\
z^{0}&z^{8}&z^{1}&z^{9}&z^{2}&z^{10}&z^{3}&z^{11}&z^{4}&z^{12}&z^{5}&z^{13}&z^{6}&z^{14}&z^{7}\\
z^{0}&z^{11}&z^{7}&z^{3}&z^{14}&z^{10}&z^{6}&z^{2}&z^{13}&z^{9}&z^{5}&z^{1}&z^{12}&z^{8}&z^{4}\\
z^{0}&z^{12}&z^{9}&z^{6}&z^{3}&z^{0}&z^{12}&z^{9}&z^{6}&z^{3}&z^{0}&z^{12}&z^{9}&z^{6}&z^{3}\\
z^{0}&z^{14}&z^{13}&z^{12}&z^{11}&z^{10}&z^{9}&z^{8}&z^{7}&z^{6}&z^{5}&z^{4}&z^{3}&z^{2}&z^{1}
\end{array}\right],
\
z=\exp(\tfrac{2\pi\rmi}{15}).
\end{equation}
Here, we have scaled $\bfPhi$ so that $\norm{\bfphi_j}^2$ is the inverse Welch bound~\eqref{equation.inverse Welch bound}, namely \smash{$s=\bracket{\frac{8(15-1)}{15-8}}^{\frac12}=4$}.
Theorem~\ref{theorem.spark bound equality} gives that a subset \smash{$\set{\bfphi_j}_{j\in\calK}$} of \smash{$\set{\bfphi_j}_{j=0}^{14}$} is a regular simplex if and only if it consists of $s+1=5$ linearly dependent vectors.
Moreover, since $\frac{n}{s+1}=\frac{15}{5}=3$,
this ETF could be a disjoint union of $3$ regular simplices.
Computing its binder reveals this is indeed the case:
\begin{equation}
\label{equation.8x15 binder}
\set{0,3,6,9,12},\
\set{1,4,7,10,13},\
\set{2,5,8,11,14}.
\end{equation}
(As we shall see in the next section, it is not a coincidence that this binder contains all cosets of the subgroup of $\bbZ_{15}$ of order $5$.)
Despite being a disjoint union of regular simplices,
this ETF is not equivalent to any Steiner ETF:
there does not exist a $\BIBD(v,k,1)$ with $v=3$ and $s=4$ since the corresponding value of $k=\frac{v-1}{s}+1=\frac32$ is not an integer.

As we now explain, every ETF \smash{$\set{\bfphi_j}_{j=1}^{n}$} that is a disjoint union of regular simplices has a Naimark complement which itself is ``almost" a disjoint union of regular simplices.
Here without loss of generality we write \smash{$\set{\bfphi_j}_{j=1}^{n}$} as \smash{$\set{\bfphi_{i,j}}_{i=1,}^{v}\,_{j=1}^{s+1}$} where $s$ is the inverse Welch bound~\eqref{equation.inverse Welch bound} and for each $i=1,\dotsc,v$ the subset \smash{$\set{\bfphi_{i,j}}_{j=1}^{s+1}$} is a regular simplex with $\norm{\bfphi_{i,j}}^2=s$ for all $j$.
For each $i$, Theorem~\ref{theorem.triple products} gives that a Naimark complement \smash{$\set{c_{i,i}}_{j=1}^{s+1}$} of \smash{$\set{\bfphi_{i,j}}_{j=1}^{s+1}$} satisfies~\eqref{equation.regular simplex Naimark complement}.
By absorbing these scalars into \smash{$\set{\bfphi_{i,j}}_{j=1}^{s+1}$}, we have without loss of generality that $c_{i,j}=1$ for all $j$,
namely that \smash{$\sum_{j=1}^{s+1}\bfphi_{i,j}=\bfzero$} or equivalently that the synthesis operator $\bfPhi_i$ of \smash{$\set{\bfphi_{i,j}}_{j=1}^{s+1}$} satisfies $\bfPhi_i^*\bfPhi_i^{}=(s+1)\bfI-\bfJ$.
These facts in mind, we have the following result:

\begin{theorem}
\label{theorem.Naimark complement of union of simplices}
Let \smash{$\set{\bfphi_{i,j}}_{i=1,}^{v}\,_{j=1}^{s+1}$} be an ETF for a $d$-dimensional space where,
for any $i=1,\dotsc,v$, \smash{$\set{\bfphi_{i,j}}_{j=1}^{s+1}$} is a regular simplex.
Then $d$ is necessarily
\begin{equation}
\label{equation.dimension of union of simplices}
d=\tfrac{vs^2}{v+s-1}.
\end{equation}
Moreover, without loss of generality assuming that for each $i=1,\dotsc,v$ we have \smash{$\sum_{j=1}^{s+1}\bfphi_{i,j}=\bfzero$} and $\norm{\bfphi_{i,j}}^2=s$ for all $j$,
then
\smash{$\set{\bfphi_{i,j}}_{i=1,}^{v}\,_{j=1}^{s+1}$} has a Naimark complement of the form:
\begin{equation}
\label{equation.Naimark complement of union of simplices}
\set{(\tfrac{v+s-1}{s})^{\frac12}\bfdelta_i\oplus(\tfrac{v-1}{s})^{\frac12}\bfpsi_{i,j}}_{i=1,}^{v}\,_{j=1}^{s+1}
\subseteq\bbF^v\oplus\bbF^{vs-d}
\end{equation}
where for any $i=1,\dotsc,v$, $\bfdelta_i$ is the $i$th standard basis element of $\bbF^v$ and, for any $i=1,\dotsc,v$, $\set{\bfpsi_{i,j}}_{j=1}^{s+1}$ is a regular simplex with \smash{$\sum_{j=1}^{s+1}\bfpsi_{i,j}=\bfzero$} and $\norm{\bfpsi_{i,j}}^2=s$ for all $j$.

Here, \smash{$\set{\bfpsi_{i,j}}_{i=1,}^{v}\,_{j=1}^{s+1}$} is an equal norm tight frame with
\begin{equation}
\label{equation.inner products of partial Naimark complement}
\abs{\ip{\bfpsi_{i,j}}{\bfpsi_{i',j'}}}
=\left\{\begin{array}{cl}
s,&i=i',\ j=j',\\
1,&i=i',\ j\neq j',\\
\tfrac{s}{v-1},&i\neq i'.
\end{array}\right.
\end{equation}
\end{theorem}

\begin{proof}
Letting $n=v(s+1)$,
note that by Theorem~\ref{theorem.spark bound equality},
the value of $s$ is necessarily~\eqref{equation.inverse Welch bound}.
Squaring~\eqref{equation.inverse Welch bound} and then solving for $d$ gives~\eqref{equation.dimension of union of simplices}.
Continuing,
for any $i=1,\dotsc,v$, let $\bfPhi_i$ be the synthesis operator of \smash{$\set{\bfphi_{i,j}}_{j=1}^{s+1}$}.
Letting $\bfone$ be the all-ones vector in $\bbF^{s+1}$,
\smash{$\bfPhi_i\bfone=\sum_{j=1}^{s+1}\bfphi_j=\bfzero$} for all $i$.
Moreover, Theorem~\ref{theorem.triple products} gives $\bfPhi_i^*\bfPhi_i^{}=(s+1)\bfI-\bfone\bfone^*$ for all $i$.
We regard the $n\times n$ Gram matrix $\bfPhi^*\bfPhi$ of \smash{$\set{\bfphi_{i,j}}_{i=1,}^{v}\,_{j=1}^{s+1}$} as a block matrix:
for any $i,i'=1,\dotsc,v$, its $(i,i')$th block is the $(s+1)\times(s+1)$ matrix $\bfPhi_i^*\bfPhi_{i'}^{}$.
Since \smash{$\set{\bfphi_{i,j}}_{i=1,}^{v}\,_{j=1}^{s+1}$} is a tight frame for its span,
$(\bfPhi^*\bfPhi)^2=a\bfPhi^*\bfPhi$ where \smash{$a=\frac{ns}d$}.

Now consider a sequence which consists of $s+1$ scaled copies of each $\bfdelta_i\in\bbF^v$:
\begin{equation}
\label{equation.proof of Naimark complement of union of simplices 1}
\set{(\tfrac{vs}{d})^{\frac12}\bfdelta_i}_{i=1,}^{v}\,_{j=1}^{s+1}
=\set{(\tfrac{v+s-1}{s})^{\frac12}\bfdelta_i}_{i=1,}^{v}\,_{j=1}^{s+1}.
\end{equation}
The square of its Gram matrix $\tfrac{vs}{d}(\bfI\otimes\bfone\bfone^*)$ is
$[\tfrac{vs}{d}(\bfI\otimes\bfone\bfone^*)]^2
=\tfrac{ns}{d}\tfrac{vs}{d}(\bfI\otimes\bfone\bfone^*)
=a\tfrac{vs}{d}(\bfI\otimes\bfone\bfone^*)$
meaning~\eqref{equation.proof of Naimark complement of union of simplices 1} is an $a$-tight frame for its span, namely for $\bbF^v$.
Moreover,
$\bfPhi^*\bfPhi(\bfI\otimes\bfone\bfone^*)=\bfzero$:
for any $i,i'=1,\dotsc,v$,
\begin{equation*}
[\bfPhi^*\bfPhi(\bfI\otimes\bfone\bfone^*)](i,i')
=\sum_{i''=1}^v(\bfPhi^*\bfPhi)(i,i'')(\bfI\otimes\bfone\bfone^*)(i'',i')
=\bfPhi_i^*\bfPhi_{i'}^{}\bfone\bfone^*
=\bfPhi_i^*\bfzero\bfone^*
=\bfzero.
\end{equation*}
Together, these facts imply that
\begin{equation}
\label{equation.proof of Naimark complement of union of simplices 2}
\set{\bfphi_{i,j}\oplus(\tfrac{vs}{d})^{\frac12}\bfdelta_i}_{i=1,}^{v}\,_{j=1}^{s+1}
\end{equation}
is an $a$-tight frame for its span:
squaring the Gram matrix
$\bfPhi^*\bfPhi+\tfrac{vs}{d}(\bfI\otimes\bfone\bfone^*)$
of~\eqref{equation.proof of Naimark complement of union of simplices 2} gives
$[\bfPhi^*\bfPhi+\tfrac{vs}{d}(\bfI\otimes\bfone\bfone^*)]^2
=(\bfPhi^*\bfPhi)^2+[\tfrac{vs}{d}(\bfI\otimes\bfone\bfone^*)]^2
=a[\bfPhi^*\bfPhi+\tfrac{vs}{d}(\bfI\otimes\bfone\bfone^*)]$.
Moreover, by~\eqref{equation.dimension of ETF span} the span of \eqref{equation.proof of Naimark complement of union of simplices 2} has dimension
\smash{$\frac{n}{a}\norm{\bfphi_{i,j}\oplus(\tfrac{vs}{d})^{\frac12}\bfdelta_i}^2
=\frac{d}{s}(s+\tfrac{vs}{d})
=d+v$}.

Now let \smash{$\set{\hat{\bfpsi}_{i,j}}_{i=1,}^{v}\,_{j=1}^{s+1}$} be any Naimark complement for~\eqref{equation.proof of Naimark complement of union of simplices 2} in
$\bbF^{v(s+1)-(d+v)}=\bbF^{vs-d}$.
Since $\bfPhi_i^*\bfPhi_i^{}=(s+1)\bfI-\bfone\bfone^*$,
every diagonal block of the Gram matrix of \smash{$\set{\hat{\bfpsi}_{i,j}}_{i=1,}^{v}\,_{j=1}^{s+1}$} is
\begin{equation*}
\hat{\bfPsi}_i^*\hat{\bfPsi}_i^{}
=a\bfI-\bfPhi_i^*\bfPhi_i^{}-\tfrac{vs}{d}\bfone\bfone^*
=(a-s-1)\bfI-(\tfrac{vs}{d}-1)\bfone\bfone^*
=\tfrac{vs-d}{d}[\tfrac{d(a-s-1)}{vs-d}\bfI-\bfone\bfone^*].
\end{equation*}
Here, note $d(a-s-1)=ns-d(s+1)=(vs-d)(s+1)$, and so
$\hat{\bfPsi}_i^*\hat{\bfPsi}_i^{}=\tfrac{vs-d}{d}[(s+1)\bfI-\bfone\bfone^*]$.
In particular, letting
\smash{$\bfpsi_{i,j}=(\tfrac{d}{vs-d})^{\frac12}\hat{\bfpsi}_{i,j}=(\tfrac{s}{v-1})^{\frac12}\hat{\bfpsi}_{i,j}$} for all $i$ and $j$,
then for each $i$ we have $\bfPsi_i^*\bfPsi_i^{}=(s+1)\bfI-\bfone\bfone^*$.
This means that for any $i$, \smash{$\set{\bfpsi_{i,j}}_{j=1}^{s+1}$} is a regular simplex that satisfies \smash{$\sum_{j=1}^{s+1}\bfpsi_{i,j}=\bfzero$}, $\norm{\bfpsi_{i,j}}^2=s$ for all $j$,
and $\ip{\bfpsi_{i,j}}{\bfpsi_{i,j'}}=1$ for all $j\neq j'$.

To conclude, note that since
$\set{\hat{\bfpsi}_{i,j}}_{i=1,}^{v}\,_{j=1}^{s+1}
=\set{(\tfrac{v-1}{s})^{\frac12}\bfpsi_{i,j}}_{i=1,}^{v}\,_{j=1}^{s+1}$ is a Naimark complement of~\eqref{equation.proof of Naimark complement of union of simplices 2} where
\smash{$\set{\bfphi_{i,j}}_{i=1,}^{v}\,_{j=1}^{s+1}$} and
\eqref{equation.proof of Naimark complement of union of simplices 1} are $a$-tight frames for their spans,
then it is also an $a$-tight frame for its span.
Also,
$\bfPhi^*\bfPhi+\tfrac{v+s-1}{s}(\bfI\otimes\bfone\bfone^*)+\tfrac{v-1}{s}\bfPsi^*\bfPsi=a\bfI$,
meaning the direct sum of any two of these $a$-tight frames is a Naimark complement of the other one.
In particular, \eqref{equation.Naimark complement of union of simplices} is a Naimark complement of \smash{$\set{\bfphi_{i,j}}_{i=1,}^{v}\,_{j=1}^{s+1}$}.
Moreover, whenever $i\neq i'$, we have $\ip{\bfphi_{i,j}}{\bfphi_{i',j'}}=-\tfrac{v-1}{s}\ip{\bfpsi_{i,j}}{\bfpsi_{i',j'}}$
for all $j,j'$, implying \smash{$\abs{\ip{\bfpsi_{i,j}}{\bfpsi_{i',j'}}}=\tfrac{s}{v-1}$} in this case.
\end{proof}

We note that Theorem~1 of the recent paper~\cite{FickusJMP19} constructs an explicit Naimark complement of any Steiner ETF.
This Naimark complement is of the form~\eqref{equation.Naimark complement of union of simplices}
for a particular choice of~\smash{$\set{\bfpsi_{i,j}}_{i=1,}^{v}\,_{j=1}^{s+1}$}.
There, that result is used to show that certain flat ETFs have a flat Naimark complement.
Here, we observe a fact not observed there, namely that
$\set{\bfpsi_{i,j}}_{j=1}^{s+1}$ is a regular simplex for any $i$.
More significantly,
Theorem~\ref{theorem.Naimark complement of union of simplices} applies more generally to any ETF that is a disjoint union of regular simplices, like~\eqref{equation.8x15 ETF}.
To understand this result in the context of this example,
note that since~\eqref{equation.8x15 ETF} is formed by taking the rows of the DFT of size $15$ indexed by $\set{1,2,3,4,8,11,12,14}$,
it has a natural Naimark complement that consists of the remaining $7$ rows.
We partition those rows into the $v=3$ that are indexed by the $3$-element subgroup $\set{0,5,10}$ of $\bbZ_{15}$, namely
\begin{equation}
\label{equation.8x15 ETF first part of Naimark}
\tfrac1{\sqrt{2}}\left[\begin{array}{lllllllllllllll}
z^{0}&z^{0}&z^{0}&z^{0}&z^{0}&z^{0}&z^{0}&z^{0}&z^{0}&z^{0}&z^{0}&z^{0}&z^{0}&z^{0}&z^{0}\\
z^{0}&z^{5}&z^{10}&z^{0}&z^{5}&z^{10}&z^{0}&z^{5}&z^{10}&z^{0}&z^{5}&z^{10}&z^{0}&z^{5}&z^{10}\\
z^{0}&z^{10}&z^{5}&z^{0}&z^{10}&z^{5}&z^{0}&z^{10}&z^{5}&z^{0}&z^{10}&z^{5}&z^{0}&z^{10}&z^{5}\\
\end{array}\right],
\end{equation}
and the remaining $vs-d=4$ indexed by $\set{6,7,9,13}$:
\begin{equation}
\label{equation.8x15 ETF second part of Naimark}
\tfrac1{\sqrt{2}}\bfPsi
=\tfrac1{\sqrt{2}}\left[\begin{array}{lllllllllllllll}
z^{0}&z^{6}&z^{12}&z^{3}&z^{9}&z^{0}&z^{6}&z^{12}&z^{3}&z^{9}&z^{0}&z^{6}&z^{12}&z^{3}&z^{9}\\
z^{0}&z^{7}&z^{14}&z^{6}&z^{13}&z^{5}&z^{12}&z^{4}&z^{11}&z^{3}&z^{10}&z^{2}&z^{9}&z^{1}&z^{8}\\
z^{0}&z^{9}&z^{3}&z^{12}&z^{6}&z^{0}&z^{9}&z^{3}&z^{12}&z^{6}&z^{0}&z^{9}&z^{3}&z^{12}&z^{6}\\
z^{0}&z^{13}&z^{11}&z^{9}&z^{7}&z^{5}&z^{3}&z^{1}&z^{14}&z^{12}&z^{10}&z^{8}&z^{6}&z^{4}&z^{2}
\end{array}\right].
\end{equation}
(Here, we have carried over the same scaling factor that we applied to~\eqref{equation.8x15 ETF}.)
We apply a unitary operator to this Naimark complement to obtain another one that is of the form~\eqref{equation.Naimark complement of union of simplices}.
Specifically, applying a unitary DFT of size $3$ to~\eqref{equation.8x15 ETF first part of Naimark} transforms it into a matrix whose columns are of the form \smash{$\set{(\tfrac{v+s-1}{s})^{\frac12}\bfdelta_i}_{i=1,}^{v}\,_{j=1}^{s+1}$},
ordered in a way so that each row of this matrix indicates a member of the partition of $\bbZ_{15}$ formed by the binder~\eqref{equation.8x15 binder}, namely a regular simplex in~\eqref{equation.8x15 ETF}:
\begin{equation*}
\tfrac{\sqrt{3}}{\sqrt{2}}\left[\begin{array}{lllllllllllllll}
1&0&0&1&0&0&1&0&0&1&0&0&1&0&0\\
0&1&0&0&1&0&0&1&0&0&1&0&0&1&0\\
0&0&1&0&0&1&0&0&1&0&0&1&0&0&1
\end{array}\right].
\end{equation*}
When ordered in the same way, we may thus regard the columns of~\eqref{equation.8x15 ETF second part of Naimark} as \smash{$\set{(\tfrac{v-1}{s})^{\frac12}\bfpsi_{i,j}}_{i=1,}^{v}\,_{j=1}^{s+1}$} where,
for any $i=1,\dotsc,v$, $\set{\bfpsi_{i,j}}_{j=1}^{s+1}$ is a regular simplex.
That is,
letting $\set{\bfpsi_j}_{j=0}^{14}$ be the columns of the matrix $\bfPsi$ given in~\eqref{equation.8x15 ETF second part of Naimark} in standard order,
we have that $\set{\bfpsi_j}_{j\in\calK}$ is a regular simplex where $\calK$ is any one of the three sets~\eqref{equation.8x15 binder}.

We also note that, being a generalization of Steiner ETFs,
it is natural to take any ETF that is a disjoint union of $v$ regular $s$-simplices,
and define the associated parameters $r:=s$, $k:=\frac{v+r-1}{r}$ and $b:=\frac{v}{k}r$.
Doing so means that $v$, $k$, $r$ and $b$ satisfy $v-1=r(k-1)$ and $vr=bk$, namely the fundamental relationships of the parameters of any $\BIBD(v,k,1)$.
However, in this general setting, the parameter $k$ is not necessarily an integer.
Indeed, \eqref{equation.8x15 ETF} has $k=\frac{3+4-1}{4}=\frac32$.
It is therefore remarkable that $b$ is an integer in general: by~\eqref{equation.dimension of union of simplices},
\smash{$b=\tfrac{v}{k}r=\tfrac{vr^2}{v+r-1}=\tfrac{vs^2}{v+s-1}=d$}.
This derived parameter $k$ is useful in characterizing when \smash{$\set{\bfpsi_{i,j}}_{i=1,}^{v}\,_{j=1}^{s+1}$} given by Theorem~\ref{theorem.Naimark complement of union of simplices} is itself an ETF:
by~\eqref{equation.inner products of partial Naimark complement},
this occurs if and only if $\frac s{v-1}=1$, namely if and only if $k=2$.
In particular, \smash{$\set{\bfpsi_{i,j}}_{i=1,}^{v}\,_{j=1}^{s+1}$} is an ETF whenever
\smash{$\set{\bfphi_{i,j}}_{i=1,}^{v}\,_{j=1}^{s+1}$} is a Steiner ETF arising from the trivial $\BIBD(v,2,1)$ that consists of all $2$-element subsets of a $v$-element set.
If we instead have $k\neq 2$,
then~\eqref{equation.inner products of partial Naimark complement} implies that \smash{$\set{\bfpsi_{i,j}}_{i=1,}^{v}\,_{j=1}^{s+1}$} is a \textit{biangular tight frame}~\cite{BargGOY15,BodmannH16,CasazzaFHT16,HaasCTC17,FickusJM18}.

\subsection{A connection to equichordal tight fusion frames}

We conclude this section by discussing another property of any ETF that happens to be a disjoint union of regular simplices: it turns out that the subspaces spanned by these simplices form a type of optimal packing known as an \textit{equichordal tight fusion frame} (ECTFF).
To elaborate, a sequence $\set{\calU_i}_{i=1}^{v}$ of $s$-dimensional subspaces of a $d$-dimensional Hilbert space $\bbH$ is said to be a \textit{tight fusion frame} for $\bbH$ if there exists $a>0$ such that $\sum_{i=1}^{v}\bfP_i=a\bfI$,
where for each $i$, $\bfP_i$ is the orthogonal projection operator onto $\calU_i$
Here, $vs=\Tr(\sum_{i=1}^{v}\bfP_i)=\Tr(a\bfI)=ad$,
and so we necessarily have $a=\frac{vs}{d}$.
Moreover, the Welch bound naturally generalizes in this setting:
for any $s$-dimensional subspaces $\set{\calU_i}_{i=1}^{n}$ of $\bbH$,
\begin{equation*}
0
\leq\biggnorm{\sum_{i=1}^{v}\bfP_i-\tfrac{vs}{d}\bfI}_{\Fro}^2
=\sum_{i=1}^{v}\sum_{\substack{i'=1\\i'\neq i}}^{v}\Tr(\bfP_i\bfP_{i'})-\tfrac{vs(vs-d)}{d}
\leq v(v-1)\max_{i\neq i'}\Tr(\bfP_i\bfP_{i'})-\tfrac{vs(vs-d)}{d}.
\end{equation*}
As such, for any $s$-dimensional subspaces $\set{\calU_i}_{i=1}^{n}$ of $\bbH$,
\begin{equation}
\label{equation.generalized Welch bound}
\tfrac{s(vs-d)}{d(v-1)}
\leq\max_{i\neq i'}\Tr(\bfP_i\bfP_{i'}),
\end{equation}
where equality holds if and only if $\set{\calU_i}_{i=1}^{v}$ is an ECTFF,
namely a tight fusion frame with the additional property that $\Tr(\bfP_i\bfP_{i'})$ is constant over all $i\neq i'$.
Such tight fusion frames are called \textit{equichordal} since letting $\set{\theta_{i,i';j}}_{j=1}^{s}$ be the \textit{principal angles} between $\calU_i$ and $\calU_{i'}$, the \textit{chordal distance} between them is given by
\smash{$[\dist(\calU_i,\calU_{i'})]^2
:=\tfrac1{2}\norm{\bfP_i-\bfP_{i'}}_{\Fro}^2
=s-\Tr(\bfP_i\bfP_{i'})
=\sum_{j=1}^{s}\sin^2(\theta_{i,i';j})$}.

In the special case where $s=1$, \eqref{equation.generalized Welch bound} reduces to the Welch bound~\eqref{equation.Welch bound}: for each $i$, $\bfP_i=\bfphi_i^{}\bfphi_i^*$ for any unit vector $\bfphi_i$ in the line $\calU_i$, and so
$\Tr(\bfP_i\bfP_{i'})=\abs{\ip{\bfphi_i}{\bfphi_{i'}}}^2$.
Moreover, for any $s$ the generalized Welch bound~\eqref{equation.generalized Welch bound} is equivalent to the \textit{simplex bound}~\cite{ConwayHS96,KutyniokPCL09}:
\begin{equation*}
\min_{i\neq i'}[\dist(\calU_i,\calU_{i'})]^2\leq\tfrac{s(d-s)}{d}\tfrac{v}{v-1}.
\end{equation*}
Because of this, every ECTFF is an optimal packing of $s$-dimensional subspaces $\set{\calU_i}_{i=1}^{v}$ of $\bbH$: the smallest chordal distance between any pair of their subspaces is as large as possible.

Various constructions of ECTFFs are known in the literature~\cite{Zauner99,KutyniokPCL09,BojarovskaP15,King16}.
In particular,~\cite{Zauner99} gives that every $\BIBD(v,k,\lambda)$ generates an ECTFF $\set{\calU_i}_{i=1}^{v}$ of subspaces of dimension \smash{$r=\lambda\frac{v-1}{k-1}$} of $\bbF^d$ where $d=b=\frac{v}{k}r$.
Here, each subspace $\calU_i$ is the range of the embedding operator $\bfE_i$ arising from the $i$th column of the corresponding $b\times v$ incidence matrix $\bfX$.
For any $i\neq i'$,
$\bfE_i^*\bfE_{i'}^{}$ is a $\set{0,1}$-valued matrix with exactly $\lambda$ nonzero entries, and so
$\Tr(\bfP_i\bfP_{i'})
=\Tr(\bfE_i^{}\bfE_i^*\bfE_{i'}^{}\bfE_{i'}^*)
=\norm{\bfE_i^*\bfE_{i'}^{}}_{\Fro}^2
=\lambda$.
When $\lambda=1$, the subspaces are the spans of the regular simplices that form the corresponding Steiner ETF.
We now generalize that result, showing that any ETF that is a disjoint union of regular simplices generates an ECTFF in this same fashion:

\begin{theorem}
\label{theorem.equichordal}
Let \smash{$\set{\bfphi_{i,j}}_{i=1,}^{v}\,_{j=1}^{s+1}$} be an ETF for a $d$-dimensional space $\bbH$ where \smash{$\set{\bfphi_{i,j}}_{j=1}^{s+1}$} is a regular simplex for each $i=1,\dotsc,v$.
Then \smash{$\set{\calU_i}_{i=1}^{v}$} is an ECTFF for $\bbH$ where  \smash{$\calU_i:=\Span\set{\bfphi_{i,j}}_{j=1}^{s+1}$}.
\end{theorem}

\begin{proof}
Without loss of generality, let \smash{$\norm{\bfphi_{i,j}}^2=\frac{s}{s+1}$} for all $i$ and $j$.
For each $i$, this implies the regular simplex \smash{$\set{\bfphi_{i,j}}_{j=1}^{s+1}$} has tight frame constant \smash{$\frac{s+1}{s}\norm{\bfphi_{i,j}}^2=1$}.
As such, the orthogonal projection operator onto $\calU_i$ is $\bfP_i=\bfPhi_i^{}\bfPhi_i^*$ where $\bfPhi_i$ is the synthesis operator of $\set{\bfphi_{i,j}}_{j=1}^{s+1}$.
For any $i\neq i'$,
$\bfPhi_i^*\bfPhi_{i'}^{}$ is the $(i,i')$th block of the Gram matrix of
\smash{$\set{\bfphi_{i,j}}_{i=1,}^{v}\,_{j=1}^{s+1}$} and so
\begin{align*}
\Tr(\bfP_i\bfP_{i'})
=\Tr(\bfPhi_i^{}\bfPhi_i^*\bfPhi_{i'}^{}\bfPhi_{i'}^*)
=\norm{\bfPhi_i^*\bfPhi_{i'}^{}}_{\Fro}^2
=\sum_{j=1}^{s+1}\sum_{j'=1}^{s+1}\abs{\ip{\bfphi_{i,j}}{\bfphi_{i',j'}}}^2.
\end{align*}
Here, since \smash{$\set{\bfphi_{i,j}}_{i=1,}^{v}\,_{j=1}^{s+1}$} is an ETF with coherence $\frac1s$,
$\abs{\ip{\bfphi_{i,j}}{\bfphi_{i',j'}}}=\frac1s\norm{\bfphi_{i,j}}^2=\frac1s\frac{s}{s+1}=\frac1{s+1}$
whenever $(i,j)\neq(i',j')$.
In particular, $\Tr(\bfP_i\bfP_{i'})=1$.
This means $\set{\calU_i}_{i=1}^{v}$ achieves equality in the generalized Welch bound~\eqref{equation.generalized Welch bound}:
by \eqref{equation.dimension of union of simplices} of Theorem~\ref{theorem.Naimark complement of union of simplices},
the left hand side of~\eqref{equation.generalized Welch bound} is
\begin{equation*}
\tfrac{s(vs-d)}{d(v-1)}
=\tfrac{s}{v-1}(\tfrac{vs}{d}-1)
=\tfrac{s}{v-1}[\tfrac{vs(v+s-1)}{vs^2}-1]
=1.\qedhere
\end{equation*}
\end{proof}

For example, applying Theorem~\ref{theorem.equichordal} to~\eqref{equation.8x15 ETF} yields an ECTFF for $\bbC^8$ that consists of $v=3$ subspaces of dimension $s=4$.
In the next section, we generalize this construction, showing that several known infinite families of harmonic ETFs are indeed disjoint unions of simplices, and so generate ECTFFs.
In passing, we recall that for any ETF that is a disjoint union of regular simplices,
Theorem~\ref{theorem.Naimark complement of union of simplices} gives a Naimark complement for it of the form~\eqref{equation.Naimark complement of union of simplices} where, for each $i=1,\dotsc,v$, \smash{$\set{\bfpsi_{i,j}}_{j=1}^{s+1}$} is a regular simplex.
Using techniques similar to those above along with~\eqref{equation.inner products of partial Naimark complement}, it is straightforward to show that $\set{\calV_i}_{i=1}^{v}$, \smash{$\calV_i:=\Span\set{\bfpsi_{i,j}}_{j=1}^{s+1}$} defines an ECTFF of $s$-dimensional subspaces of $\bbF^{vs-d}$.
We do not pursue this here, since the existence of an ECTFF of these parameters also follows immediately by taking fusion-frame-based Naimark complements~\cite{CasazzaFMWZ11} of the ECTFFs given by Theorem~\ref{theorem.equichordal}.

\section{Families of equiangular tight frames that contain regular simplices}

We now apply the concepts and results from the previous sections to several known families of ETFs.
In particular, given an ETF, there are three things we would like to know.
First, does it contain a regular simplex?
That is, is its binder nonempty?
By Theorem~\ref{theorem.spark bound equality}, this occurs if and only if the ETF achieves equality in~\eqref{equation.spark bound}.
Second, is it a disjoint union of regular simplices?
That is, does its binder contain a partition of $[n]$?
If so, it has a Naimark complement that is ``almost" a disjoint union of simplices,
and it generates an ECTFF; see Theorems~\ref{theorem.Naimark complement of union of simplices} and~\ref{theorem.equichordal}, respectively.
Third, does its binder contain a BIBD?
If so, by Theorem~\ref{theorem.phased binder}, it has a Naimark complement that is a phased BIBD ETF.
In particular, we now discuss the degree to which we can answer these questions for harmonic ETFs~\cite{StrohmerH03,XiaZG05,DingF07}, Steiner ETFs~\cite{FickusMT12}, ETFs arising from hyperovals~\cite{FickusMJ16}, Tremain ETFs~\cite{FickusJM18}, phased BIBD ETFs~\cite{FickusJMPW19} and \textit{symmetric informationally-complete positive operator-valued measures} (SIC-POVMs)~\cite{Zauner99,RenesBSC04}.
Here, our most significant contributions are made in the context of harmonic ETFs, and so we dwell on them the most.

\subsection{Harmonic ETFs}

A subset $\calD$ of a finite abelian group $\calG$ is a \textit{difference set} if every nonzero element of $\calG$ arises as a difference of elements in $\calD$ the same number of times, namely if $\#\set{(d,d')\in\calD\times\calD: g=d-d'}$  is constant over $g\in\calG$, $g\neq 0$.
Restricting the characters of $\calG$ to a difference set $\calD$ gives a so-called \textit{harmonic} ETF~\cite{StrohmerH03,XiaZG05,DingF07}.

A harmonic ETF need not necessarily contain a regular simplex.
Indeed, the inverse Welch bound~\eqref{equation.inverse Welch bound} of a harmonic ETF need not even be an integer.
For example, the quadratic residues $\set{1,2,4}$ are a difference set in $\bbZ_7$,
and the corresponding $7$-vector ETF for $\bbC^3$ has inverse Welch bound~\smash{$\frac{3}{\sqrt{2}}$}.
That said, some ETFs do contain regular simplices.
For example, every harmonic ETF that arises from a McFarland difference set is equivalent to a Steiner ETF that arises from an affine geometry, and so is a disjoint union of regular simplices~\cite{JasperMF14}.
Moreover, as seen in the previous section, there are harmonic ETFs that are not equivalent to any Steiner ETF and yet are a disjoint union of regular simplices,
like~\eqref{equation.8x15 ETF} whose binder is~\eqref{equation.8x15 binder}.
As a generalization of~\eqref{equation.8x15 ETF}, we now show that a harmonic ETF is a disjoint union of regular simplices whenever its difference set is disjoint from a subgroup of the underlying group of the appropriate order:

\begin{theorem}
\label{theorem.harmonic ETFs}
Let $\calD$ be a $d$-element difference set in an abelian group $\calG$ of order $n>d$,
and let $s$ be the inverse Welch bound~\eqref{equation.inverse Welch bound}.
If there exists a subgroup $\calH$ of $\calG$ of order $v:=\frac{\abs{\calG}}{s+1}$ that is disjoint from $\calD$, then the corresponding harmonic ETF is a disjoint union of $v$ regular $s$-simplices.
\end{theorem}

\begin{proof}
Let $\Gamma$ be the Pontryagin dual of $\calG$,
namely the group of all (continuous) homomorphisms from $\calG$ into $\bbT=\set{z\in\bbC: \abs{z}=1}$.
Let $\bfPsi$ be the character table of $\calG$, that is, $\bfPsi(g,\gamma):=\gamma(g)$ for all $g\in\calG$, $\gamma\in\Gamma$.
Let $\bfPhi$ be the $d\times n$ submatrix of $\bfPsi$ whose rows are indexed by $\calD$.
Since $\calD$ is a difference set, the columns $\set{\bfphi_\gamma}_{\gamma\in\Gamma}$ of $\bfPhi$ form an ETF for $\bbC^\calD\cong\bbC^d$~\cite{XiaZG05,DingF07}.

Since $\bfPsi$ is a complex Hadamard matrix, the $d$ rows of $\bfPsi$ that constitute $\bfPhi$ are orthogonal to the remaining $n-d$ rows of $\bfPsi$.
Since $\calD\cap\calH=\emptyset$, this includes the $v$ rows of $\bfPsi$ that are indexed by $\calH$.
Moreover, the finite Poisson summation formula gives that these $\calH$-indexed rows sum to $v$ times the characteristic function $\bfone_{\calH^\perp}$ of the annihilator $\calH^\perp=\set{\gamma\in\Gamma: \gamma(h)=1, \forall\, h\in\calH}$ of $\calH$.
Here $\calH^\perp$ is a subgroup of $\Gamma$ which is isomorphic to $\calG/\calH$ and so has order $s+1$.
As such, $\bfone_{\calH^\perp}$ is orthogonal to the rows of $\bfPhi$,
giving \smash{$\sum_{\gamma\in\calH^\perp}\bfphi_\gamma=\bfPhi\bfone_{\calH^\perp}=\bfzero$}.
In particular, the $s+1$ vectors $\set{\bfphi_\gamma}_{\gamma\in\calH^\perp}$ are linearly dependent,
and so Theorem~\ref{theorem.spark bound equality} implies $\set{\bfphi_\gamma}_{\gamma\in\calH^\perp}$ is a regular simplex.
In fact, this more generally implies that each coset of \smash{$\calH^\perp$} indexes a regular simplex:
for any $\gamma\in\Gamma$ and any $g\in\calG$,
\begin{equation*}
\sum_{\gamma'\in\gamma\calH^\perp}\bfphi_{\gamma'}(g)
=\sum_{\gamma'\in\gamma\calH^\perp}\gamma'(g)
=\sum_{\gamma''\in\calH^\perp}(\gamma\gamma'')(g)
=\gamma(g)\sum_{\gamma''\in\calH^\perp}\gamma''(g)
=\gamma(g)0
=0,
\end{equation*}
meaning $\set{\bfphi_{\gamma'}}_{\gamma'\in\gamma\calH^\perp}$ is linearly dependent and is thus a regular simplex.
In particular, since $\Gamma$ is a disjoint union of $v$ cosets of $\calH^\perp$,
$\set{\bfphi_\gamma}_{\gamma\in\Gamma}$ is a disjoint union of regular simplices.
\end{proof}

We emphasize that a harmonic ETF may contain more simplices than those given by this result: it only provides sufficient conditions for the vectors indexed by cosets of $\calH^\perp$ to form a simplex.
For example,
$\calD=\set{(0,0,1,0),(0,1,1,0),(0,0,0,1),(1,0,0,1),(0,0,1,1),(1,1,1,1)}$ is a McFarland difference set in $\calG=\bbZ_2^4$,
and yields an ETF $\set{\bfphi_j}_{j=1}^{16}$ for $\bbR^6$, namely the columns of:
\begin{equation}
\label{equation.6x16 harmonic ETF}
\bfPhi=\left[\begin{array}{rrrrrrrrrrrrrrrr}
1& 1& 1& 1&-1&-1&-1&-1& 1& 1& 1& 1&-1&-1&-1&-1\\
1& 1&-1&-1&-1&-1& 1& 1& 1& 1&-1&-1&-1&-1& 1& 1\\
1& 1& 1& 1& 1& 1& 1& 1&-1&-1&-1&-1&-1&-1&-1&-1\\
1&-1& 1&-1& 1&-1& 1&-1&-1& 1&-1& 1&-1& 1&-1& 1\\
1& 1& 1& 1&-1&-1&-1&-1&-1&-1&-1&-1& 1& 1& 1& 1\\
1&-1&-1& 1&-1& 1& 1&-1&-1& 1& 1&-1& 1&-1&-1& 1
\end{array}\right].
\end{equation}
Here, the inverse Welch bound is $s=3$, and there is a natural subgroup of $\bbZ_2^4$ of order $v=4$ that is disjoint from $\calD$, namely $\calH=\set{(0,0,0,0),(0,0,1,0),(0,0,0,1),(0,0,0,1)}$.
As such, Theorem~\ref{theorem.harmonic ETFs} implies that $\set{\bfphi_j}_{j=1}^{16}$ is a disjoint union of $4$ regular simplices.
(This also follows from~\cite{JasperMF14}, which gives that~\eqref{equation.6x16 harmonic ETF} is equivalent to a Steiner ETF.)
Applying BinderFinder to this ETF reveals that its binder is a $\BIBD(16,4,3)$, meaning it contains exactly $60$ regular simplices.
In some other cases, the result of Theorem~\ref{theorem.harmonic ETFs} is exact:
$\calD=\set{1,2,3,4,8,11,12,14}$ is a difference set in $\calG=\bbZ_{15}$,
and is disjoint from the subgroup $\calH=\set{0,5,10}$ of $\calG$ of order $v=3$,
meaning the corresponding harmonic ETF~\eqref{equation.8x15 ETF} is a disjoint union of $3$ regular simplices.
Applying BinderFinder reveals its binder to be~\eqref{equation.8x15 binder}, meaning these are the only regular simplices it contains.

We now apply Theorem~\ref{theorem.harmonic ETFs} to several known families of difference sets.
Here, we note that since the complement $\calD^\rmc$ of a difference set is another difference set,
we can alternatively interpret this result as saying that whenever a difference set of $\calG$ contains a subgroup of $\calG$ of appropriate order, then its Naimark complements are disjoint unions of regular simplices.

\begin{example}
Theorem~\ref{theorem.harmonic ETFs} applies to the complements of certain \textit{Singer} difference sets.
To be precise, for any prime power $q$ and integer $e\geq 2$,
let $\calG=\bbF_{\smash{q^{e+1}}}^\times/\bbF_q^\times$ be the cyclic group of order \smash{$n=\frac{q^{e+1}-1}{q-1}$}
obtained by taking the quotient of the multiplicative group of $\bbF_{q^{e+1}}$ over the multiplicative group of the base field $\bbF_q$.
A Singer difference set in $\calG$ is a set of the form $\calE=\set{\gamma\bbF_q^\times: \gamma\in\calT}$ where $\calT$ is any hyperplane of $\bbF_{q^{e+1}}$, that is, any $e$-dimensional subspace of the $(e+1)$-dimensional vector space $\bbF_{q^{e+1}}$ over $\bbF_q$.
Such difference sets can be equivalently expressed as \smash{$\set{i\in\bbZ_n: \tr(\alpha^{i+j})=0}$} where $\alpha$ is a primitive element of $\bbF_{q^{e+1}}$ and
$\tr:\bbF_{q^{e+1}}\rightarrow\bbF_q$ is the field trace~\cite{JungnickelPS07}.
Here, $j\in\bbZ_n$ is arbitrary and corresponds to a particular choice of hyperplane $\calT$.
The complement $\calD=\calE^\rmc$ of a Singer difference set has \smash{$d=\abs{\calD}=\abs{\calG}-\abs{\calE}=\frac{q^{e+1}-1}{q-1}-\frac{q^e-1}{q-1}=q^e$} elements,
implying the inverse Welch bound~\eqref{equation.inverse Welch bound} is $s=q^{\frac{e+1}2}$.
In particular, $s+1$ is a positive integer whenever $e$ is odd.
As such, from this point forward we assume $e=2f-1$ for some $f\geq 2$.
In order to apply Theorem~\ref{theorem.harmonic ETFs}, we need $\calE=\calD^c$ to contain a subgroup $\calH$ of $\calG$ of order
\begin{equation*}
v
=\tfrac{\abs{\calG}}{s+1}
=\tfrac{q^{2f}-1}{q-1}\tfrac1{q^f+1}
=\tfrac{q^f-1}{q-1}.
\end{equation*}
Since $\calG$ is a cyclic, it has exactly one such subgroup $\calH$ of this order.
In fact, since $f$ divides $2f$, $\bbF_{q^{2f}}$ contains $\bbF_{q^f}$ as a subfield,
meaning this subgroup is \smash{$\calH=\bbF_{\smash{q^f}}^\times/\bbF_{\smash{q}}^\times$}.
Moreover, this subgroup $\calH$ is contained in $\calE$ whenever its defining hyperplane $\calT$ is chosen to be any one of the several $(2f-1)$-dimensional subspaces of $\bbF_{q^{2f}}$ that contains the $f$-dimensional subspace $\bbF_{q^f}$.
In particular, we can always choose $j$ such that \smash{$\calE=\set{i\in\bbZ_n: \tr(\alpha^{i+j})=0}$} contains the subgroup of $\bbZ_n$ of order \smash{$\frac{q^f-1}{q-1}$}.

This means the complementary difference set yields a harmonic ETF of \smash{$n=\frac{q^{2f}-1}{q-1}$} vectors for \smash{$\bbC^d$} where $d=q^{2f-1}$,
where this ETF is a disjoint union of
\smash{$v=\abs{\calH}=\tfrac{q^f-1}{q-1}$} regular $s$-simplices where $s=q^f$.
These ETFs are not equivalent to Steiner ETFs since they have redundancy
\begin{equation*}
\tfrac{n}{d}
=\tfrac{q^{2f}-1}{q-1}\tfrac1{q^{2f-1}}
=1+\tfrac{q^{2f-1}-1}{q-1}\tfrac1{q^{2f-1}}
<1+\tfrac1{q-1}
\leq 2,
\end{equation*}
whereas the redundancy~\eqref{equation.redundancy of Steiner ETF} of a Steiner ETF is greater than two.

For example, when $q=3$ and $f=2$, $x^4+x+2$ is a primitive polynomial over $\bbF_3$~\cite{HansenM92}, meaning $\alpha$ generates the multiplicative group of
$\bbF_{81}
=\set{a+b\alpha+c\alpha^2+d\alpha^3: a,b,c,d\in\bbF_3,\, \alpha^4=2\alpha+1}$.
Here, the canonical hyperplane $\set{\gamma\in\bbF_{81}:0=\tr(\gamma)=\gamma+\gamma^3+\gamma^9+\gamma^{27}}$ can be computed to be
\begin{equation*}
\set{0}\cup\set{\alpha^i: i=1,2,3,5,6,9,14,15,18,20,25,27,35,41,42,43,45,46,49,54,55,58,60,65,67,75}.
\end{equation*}
(Here, one can compute $\tr(\alpha^i)$ for $i=0,1,2,3$ by taking the usual trace of $\bfA^i$ where $\bfA$ is the $4\times 4$ companion matrix of $x^4+x+2$ over $\bbF_3$.
This initializes the recursion $\tr(\alpha^{i})=2\tr(\alpha^{i-3})+\tr(\alpha^{i-4})$ that arises from $\alpha^4=2\alpha+1$.)
Removing $0$ from this hyperplane and identifying the remaining elements modulo $\bbF_3^\times=\set{1,-1}=\set{1,\alpha^{40}}$ gives the following \smash{$\frac{q^{2f-1}-1}{q-1}=13$}-element Singer difference set in \smash{$\bbZ_{40}\cong\bbF_{81}^\times/\bbF_{3}^\times$}:
\begin{equation*}
\set{1,2,3,5,6,9,14,15,18,20,25,27,35}.
\end{equation*}
The $40$ cyclic translates of this set correspond to the $40$ distinct hyperplanes in $\bbF_{81}$.
To apply Theorem~\ref{theorem.harmonic ETFs}, we want one of these translates to contain the subgroup of $\bbZ_{40}$ of order \smash{$\frac{q^f-1}{q-1}=4$},
namely $\set{0,10,20,30}$.
For example, translating by $-5$ gives the alternative Singer difference set
\begin{equation*}
\set{0,1,4,9,10,13,15,20,22,30,36,37,38}.
\end{equation*}
(We could have also translated by $-15$, $-25$ or $-35$.)
This is guaranteed to be possible since there are several $(2f-1)=3$-dimensional subspaces of $\bbF_{81}$ that contain the $f=2$-dimensional subspace
$\bbF_9=\set{0}\cup\set{\alpha^i:i=0,10,20,30,40,50,60,70}$.
Taking the remaining $27$ rows of the $40\times40$ discrete Fourier transform gives an ETF of $40$ vectors for $\bbC^{27}$ that is a disjoint union of four $9$-simplices.
Meanwhile,
applying a similar rationale in the case where $q=2$ and $f=2$ leads to~\eqref{equation.8x15 ETF},
which is a disjoint union of $3$ regular $4$-simplices.
\end{example}

\begin{example}
Theorem~\ref{theorem.harmonic ETFs} applies to \textit{McFarland} difference sets.
To elaborate, for any prime power $q$ and any positive integer $e$,
let \smash{$\set{\calU_i: i=1,\dotsc,\frac{q^{e+1}-1}{q-1}}$} denote the distinct hyperplanes of the vector space $\bbF_q^{e+1}$,
and let \smash{$\calT=\set{t_i: i=0,\dotsc,\frac{q^{e+1}-1}{q-1}}$} be any abelian group of order \smash{$\frac{q^{e+1}-1}{q-1}+1$}.
A McFarland difference set~\cite{JungnickelPS07} for the group $\calG=\calT\times\bbF_q^{e+1}$ is any set of the form
\begin{equation*}
\calH=\set{(t,u)\in\calG: t=t_i,\, u\in\calU_i\text{ for some }i=1,\dotsc,\tfrac{q^{e+1}-1}{q-1}}.
\end{equation*}
Here, the cardinalities of $\calD$ and $\calG$ are
\begin{equation*}
d=\abs{\calD}=q^e\bigparen{\tfrac{q^{e+1}-1}{q-1}},
\quad
n=\abs{\calG}=q^{e+1}\bigparen{1+\tfrac{q^{e+1}-1}{q-1}},
\end{equation*}
respectively.
At this point, a direct calculation reveals that $s=\tfrac{q^{e+1}-1}{q-1}$.
To apply Theorem~\ref{theorem.harmonic ETFs}, we thus need $\calD$ to be disjoint from a subgroup $\calH$ of $\calG$ of order
\begin{equation*}
v
=\tfrac{\abs{G}}{s+1}
=\tfrac1{\abs{\calT}}\abs{\calT\times\bbF_q^{e+1}}
=\abs{\bbF_q^{e+1}}
=q^{e+1}.
\end{equation*}
Here, a natural choice is to let $\calH=\set{0}\times\bbF_q^{e+1}$.
This subgroup is indeed disjoint from $\calD$ provided we enumerate $\calT$ so that $t_0=0$.
In this case, Theorem~\ref{theorem.harmonic ETFs} guarantees that the corresponding harmonic ETF is a union of simplices.
This is not a surprise, since every harmonic ETF arising from a McFarland difference set is known to be equivalent to a Steiner ETF~\cite{JasperMF14}.
For a specific example of these ideas, letting $q=2$ and $e=1$ leads to~\eqref{equation.6x16 harmonic ETF}, which is a disjoint union of $4$ regular $3$-simplices.
\end{example}

\begin{example}
Theorem~\ref{theorem.harmonic ETFs} applies to the complements of \textit{twin prime power} difference sets.
Here, for any odd prime power $q$ such that $q+2$ is also a prime power, let $\calG=(\bbF_q,+)\times(\bbF_{q+2},+)$,
which is a cyclic group of order $n=q(q+2)$.
The twin prime power difference set $\calE$ for $\calG$ is a set of order $\frac12(q^2+2q-1)$ that consists of those $(x,y)$ such that either (i) $y=0$ or (ii) $x$ and $y$ are both nonzero squares or (iii) $x$ and $y$ are both nonsquares~\cite{JungnickelPS07}.
Its complement $\calD$ is a difference set for $\calG$ of order $d=\abs{\calD}=\abs{\calG}-\abs{\calE}=q(q+2)-\frac12(q^2+2q-1)=\frac12(q+1)^2$,
implying its inverse Welch bound~\eqref{equation.generalized Welch bound} is $s=q+1$.
To apply Theorem~\ref{theorem.harmonic ETFs} we need $\calE$ to contain a subgroup $\calH$ of $\calG$ of order \smash{$v=\frac{\abs{\calG}}{s+1}=\frac{q(q+2)}{q+2}=q$}.
This suggests taking $\calH=\bbF_q\times\set{0}$, which is indeed contained in $\calE$.
As such, the harmonic ETF arising from the complement of a twin prime power difference set is a disjoint union of regular simplices.
These ETFs are not Steiner ETFs since $\frac nd=2\frac{q(q+2)}{(q+1)^2}=2[1-\frac1{(q+1)^2}]<2$.
\end{example}

We summarize these examples as follows:

\begin{theorem}
\label{theorem.instances of harmonic ETFs}
For any prime power $q$,
there exists an $n$-vector harmonic ETF for $\bbC^d$ that is a disjoint union of $v$ regular $s$-simplices whenever:
\begin{enumerate}
\renewcommand{\labelenumi}{(\alph{enumi})}
\item\
\smash{$d=q^{2f-1}$},
\smash{$n=\frac{q^{2f}-1}{q-1}$},
\smash{$v=\frac{q^f-1}{q-1}$} and
\smash{$s=q^f$} where $f\geq 2$;\smallskip
\item\
\smash{$d=q^e\bigparen{\tfrac{q^{e+1}-1}{q-1}}$},
\smash{$n=q^{e+1}\bigparen{1+\tfrac{q^{e+1}-1}{q-1}}$},
\smash{$v=q^{e+1}$} and
\smash{$s=\tfrac{q^{e+1}-1}{q-1}$} where $e\geq 1$;\smallskip
\item\
$d=\frac12(q+1)^2$,
$n=q(q+2)$,
$v=q$ and
$s=q+1$ provided $q+2$ is an odd prime power.
\end{enumerate}
\end{theorem}

Applying Theorem~\ref{theorem.equichordal} to any one of these ETFs gives an ECTFF.
Meanwhile, applying Theorem~\ref{theorem.Naimark complement of union of simplices} gives a Naimark complement for it that is a direct sum of the $s+1$ copies of the standard basis for $\bbF^v$ with a biangular tight frame that itself is a disjoint union of regular simplices.
In fact, a careful read of the proofs of Theorems~\ref{theorem.equichordal} and~\ref{theorem.harmonic ETFs} reveals that we can alternatively partition $\calG$ into $\calH$,
$\calD$ and $\calE\cap\calH^\rmc$, and that the corresponding $v$, $d$ and $vs-d$ rows of the character table of $\calG$ yield complementary harmonic tight frames:
the first of these frames consists of $s+1$ copies of the discrete Fourier basis for $\calH$, while the second and third are equiangular and biangular tight frames for $\bbF^d$ and $\bbF^{vs-d}$, respectively, each consisting of a disjoint union of $v$ regular $s$-simplices.

In the special case where $f=2$, the Naimark complement of the ETF in Theorem~\ref{theorem.instances of harmonic ETFs}(a) is of the
form~\eqref{equation.Naimark complement of union of simplices} where
\smash{$\set{\bfpsi_{i,j}}_{i=1,}^{v}\,_{j=1}^{s+1}$} is a disjoint union of $s$-simplices in $\bbF^{vs-d}$ where $vs-d=q^2=s$.
That is, in this case, \smash{$\set{\bfpsi_{i,j}}_{i=1,}^{v}\,_{j=1}^{s+1}$} is a biangular tight frame that consists of $v=q+1$ regular $s$-simplices for $\bbF^s$ where $s=q^2$.
These vectors thus form a ``regular simplex" analog of a mutually unbiased basis, which is a biangular tight frame that consists of orthonormal bases for $\bbF^s$.

In summary, though a harmonic ETF need not contain a regular simplex in general,
there are nevertheless a few infinite families of them that do,
and in fact are a disjoint union of regular simplices.
Less obvious is whether any of these infinite families of harmonic ETFs have the property that their binders contain a BIBD.
This is clearly not true in general since for certain harmonic ETFs, there is exactly one subgroup $\calH$ that satisfies the hypotheses of Theorem~\ref{theorem.harmonic ETFs},
and the only regular simplices it contains correspond to the cosets of $\calH$.
It turns out this happens, for example, for the harmonic ETF~\eqref{equation.8x15 binder} that arises from the complement of a Singer difference set,
as well as for the $45$-vector ETF for $\bbC^{12}$ that arises from the McFarland difference set with $q=3$ and $e=1$.
In fact, to date, the only harmonic ETFs whose binders contain BIBDs that we have discovered arise from McFarland difference sets in the special case where $q=2$.
The fact that these binders have this property is not surprising, since ETFs with these same parameters arise as the Naimark complements of phased $\BIBD(q^2,q,1)$ ETFs in the special case where $q=2^{e+1}$.
Because of this, we leave a deeper investigation of harmonic ETFs whose binders contain BIBDs for future research.

\subsection{Steiner ETFs}

As we have already seen, every Steiner ETF is by design a disjoint union of regular simplices.
In particular, they achieve equality in~\eqref{equation.spark bound}.
As mentioned earlier, applying Theorem~\ref{theorem.Naimark complement of union of simplices} to them refines a result of~\cite{FickusJMP19}, while applying Theorem~\ref{theorem.equichordal} to them recovers a known result of~\cite{Zauner99}.

What remains is the extent to which the binders of Steiner ETFs contain BIBDs.
As with harmonic ETFs, this issue seems complicated: every harmonic ETF arising from a McFarland difference set is equivalent to a Steiner ETF, and as we have already discussed, some McFarland ETFs with $q=2$ seem to have this property, while some with $q=3$ do not.
That said, our experimentation with BinderFinder indicates that one should not expect the binder of a Steiner ETF to contain a BIBD in general.

\subsection{ETFs from hyperovals}

Whenever $q$ is an even prime power, \cite{FickusMJ16} gives an ETF \smash{$\set{\bfphi_j}_{j=1}^{n}$} of $n=q(q^2+q-1)$ vectors in $\bbC^d$ where $d=q^2+q-1$.
The construction itself is a generalization of that used for Steiner ETFs,
and only applies when the underlying BIBD is a projective plane that contains a \textit{hyperoval}, namely a $\BIBD(q^2+q+1,q+1,1)$ whose vertex set contains a subset of $q+2$ points, no three collinear.
It turns out that such BIBDs can only exist when $q$ is even,
and are known to exist whenever $q$ is an even prime power.
The dual design of such a BIBD is a projective plane with the property that a subset of its vertices and blocks forms a $\BIBD(\tfrac12q(q-1),\tfrac12q,1)$.
By design, the first $n_0=\frac12q(q-1)(q+2)$ vectors in this ETF form a Steiner ETF for a subspace of $\bbC^d$ of dimension $d_0=q^2-1$ arising from this subdesign.
In particular, \smash{$\set{\bfphi_j}_{j=1}^{n_0}$} is a disjoint union of $v=\frac12q(q-1)$ regular $s$-simplices where $s=q+1$.
As such, \smash{$\set{\bfphi_j}_{j=1}^{n}$} contains regular simplices, and so achieves equality in~\eqref{equation.spark bound}.
That said,
\smash{$\tfrac{n}{s+1}=\tfrac{q(q^2+q-1)}{q+2}=q^2-q+1-\tfrac{2}{q+2}$}
is never an integer, and so \smash{$\set{\bfphi_j}_{j=1}^{n}$} is never a disjoint union of regular simplices.
It is unclear when the binder of such an ETF contains a BIBD, but this does occur at least once:
when $q=2$, the resulting $10$-vector ETF for $\bbC^5$ is equivalent to~\eqref{equation.5x10 ETF}, whose binder is a $\BIBD(10,4,2)$.

From the perspective of this paper,
a remarkable feature of this class of ETFs is that their existence is implied by the existence of other ETFs that happen to contain a regular simplex:

\begin{theorem}
\label{theorem.hyperoval}
Let $q\geq 2$ be an integer and suppose \smash{$\set{\bfpsi_j}_{j=1}^{q^2(q+2)}$} is an ETF for a space of dimension $q(q^2+q-1)$ that contains a regular simplex.
Then any Naimark complement of it contains a $q(q^2+q-1)$-vector ETF for a space of dimension $q^2+q-1$.
\end{theorem}

\begin{proof}
A direct computation reveals that the inverse Welch bound of \smash{$\set{\bfpsi_j}_{j=1}^{q^2(q+2)}$} is $q^2+q-1$.
Without loss of generality, let $\norm{\bfpsi_j}^2=q^2+q-1$ for all $j$, implying  $\ip{\bfpsi_j}{\bfpsi_{j'}}$ is unimodular for all $j\neq j'$,
and that the tight frame constant of \smash{$\set{\bfpsi_j}_{j=1}^{q^2(q+2)}$} is $a=q(q+2)$.
Without loss of generality, further assume that the regular simplex this ETF contains is its first $q(q+1)$ vectors, and that these vectors sum to zero.
In particular, by Theorem~\ref{theorem.triple products} we can assume the Gram matrix of \smash{$\set{\bfpsi_j}_{j=1}^{q^2+q}$} is
$\bfPsi_0^*\bfPsi_0^{}=q(q+1)\bfI-\bfone\bfone^*$,
where $\bfone$ is an all-ones vector of length $q(q+1)$.
As such,
the corresponding principal submatrix of the Gram matrix of any Naimark complement
\smash{$\set{\bfphi_j}_{j=1}^{q^2(q+2)}$} in $\bbF^{q(q+1)}$ of the ETF \smash{$\set{\bfpsi_j}_{j=1}^{q^2(q+2)}$} is
\begin{equation*}
\bfPhi_0^*\bfPhi_0^{}
=a\bfI-\bfPsi_0^*\bfPsi_0^{}
=q(q+2)\bfI-[q(q+1)\bfI-\bfone\bfone^*]
=q\bfI+\bfone\bfone^*.
\end{equation*}
Since this matrix is positive definite, the vectors \smash{$\set{\bfphi_j}_{j=1}^{q(q+1)}$} form a basis for $\bbF^{q(q+1)}$.
Moreover, since a Naimark complement is only unique up to unitary equivalence, we can without loss of generality take $\bfPhi_0$ to be a self-adjoint matrix that satisfies $\bfPhi_0^2=q\bfI+\bfone\bfone^*$:
\begin{equation*}
\bfPhi_0
=q^{-\frac12}\bigparen{q\bfI+\tfrac{-1\pm\sqrt{q+2}}{q+1}\bfone\bfone^*}.
\end{equation*}
As such, $\bfPhi_0\bfPhi_0^*=\bfPhi_0^2=q\bfI+\bfone\bfone^*$.
Letting $\bfPhi_1$ be the synthesis operator of \smash{$\set{\bfphi_j}_{j=q(q+1)+1}^{q^2(q+2)}$},
this means the frame operator of these vectors is
\begin{equation*}
\bfPhi_1^{}\bfPhi_1^*
=a\bfI-\bfPhi_0^{}\bfPhi_0^*
=q(q+2)\bfI-(q\bfI+\bfone\bfone^*)
=q(q+1)\bfI-\bfone\bfone^*.
\end{equation*}
Since
$(\bfPhi_1^{}\bfPhi_1^*)^2
=[q(q+1)\bfI-\bfone\bfone^*]^2
=q(q+1)[q(q+1)\bfI-\bfone\bfone^*]
=q(q+1)\bfPhi_1^{}\bfPhi_1^*$,
these vectors form a tight frame for their span, which necessarily has dimension \smash{$\frac{q(q^2+q-1)}{q(q+1)}\norm{\bfphi_j}^2=q^2+q-1$}.
Moreover, since these vectors are a subset of the ETF~\smash{$\set{\bfphi_j}_{j=1}^{q^2(q+2)}$},
they themselves are equiangular, meaning they form an ETF for their $(q^2+q-1)$-dimensional span.
\end{proof}

This result can be interpreted as a partial converse of a special case of a result of~\cite{FickusMJ16}.
Specifically, if there exist vectors \smash{$\set{\bfphi_j}_{j=1}^{q(q^2+q-1)}$} with all unimodular entries that form an ETF for
$\bfone^\perp=\set{\bfy\in\bbF^{q(q+1)}: \ip{\bfone}{\bfy}=0}$,
then Theorem~4 of~\cite{FickusMJ16} implies that the $q^2(q+2)$ columns of
\begin{equation*}
\left[\begin{array}{cc}q\bfI+\frac{\sqrt{q+2}-1}{q+1}\bfone\bfone^*&\ \bfPhi\end{array}\right]
\end{equation*}
form an ETF for $\bbF^{q(q+1)}$, and it is straightforward to verify that the first $q(q+1)$ vectors of any Naimark complement \smash{$\set{\bfpsi_j}_{j=1}^{q^2(q+2)}$} of this ETF form a regular simplex.

Theorem~\ref{theorem.hyperoval} suggests an approach for constructing ETFs of $n=q(q^2+q-1)$ vectors for $\bbF^d$ where $d=q^2+q-1$ that is alternative to that of~\cite{FickusMJ16}:
simply find an ETF of $q^2(q+2)$ vectors for a space of dimension $q(q+1)$ whose Naimark complement contains a regular simplex.
For example, to construct a $76$-vector ETF for $\bbC^{19}$, we simply need a $96$-vector ETF for $\bbC^{20}$ whose Naimark complement contains a regular simplex.
The techniques of~\cite{FickusMJ16} show that such ETFs necessarily exist,
while~\cite{AzarijaM15,AzarijaM16} show that such ETFs are necessarily complex.
Moreover, there is a well-known $96$-vector harmonic ETF for $\bbC^{20}$ that corresponds to a McFarland difference set with $q=4$ and $e=1$.
Remarkably however, applying BinderFinder to the Naimark complement of that ETF---a process which takes our current Matlab implementation~\cite{FickusJKM17} around an hour to run on current desktop computers---reveals that it contains no regular simplices.
There are thus at least two nonequivalent $96$-vector ETFs for $\bbC^{20}$:
one that contains a $76$-vector ETF for $\bbC^{19}$ and whose Naimark complements contain a regular simplex, and a harmonic ETF with these same parameters that does neither of these things.
We leave a deeper investigation of these ideas for future work.

\subsection{Tremain ETFs}

A $\BIBD(v,3,1)$ is known as a \textit{Steiner triple system},
and such a BIBD is known to exist if and only if $v \equiv 1,3 \bmod 6$ and $v\geq 7$.
As shown in~\cite{FickusJMP18}, every Steiner ETF arising from such a BIBD can be modified to produce an ETF of $n=\frac12(v+1)(v+2)$ vectors for a space of dimension $d=\frac16(v+2)(v+3)$.
Here as before, the underlying Steiner ETF is \smash{$\set{\bfE_j\bff_j}_{i=1}^{v}\,_{j=1}^{r+1}$}
where $\set{\bff_j}_{j=1}^{r+1}$ is a unimodular (flat) regular simplex in $\bbF^r$,
and each $\bfE_i$ is a $b\times r$ embedding matrix arising from the corresponding column of the BIBD's $b\times v$ incidence matrix $\bfX$.
Here, since $k=3$ we have $r=\frac12(v-1)$ and $b=\frac16v(v-1)$.
The corresponding Tremain ETF is then the vectors
\begin{equation}
\label{equation.definition of Tremain ETF}
\set{\bfE_i\bff_j\oplus\sqrt{2}\bfdelta_i\oplus0}_{i=1,}^{v}\,_{j=1}^{r+1}
\cup\set{\bfzero\oplus\tfrac1{\sqrt{2}}\bfg_j\oplus\tfrac{\sqrt{3}}{\sqrt{2}}c_j}_{j=1}^{v+1}
\end{equation}
in $\bbF^b\oplus\bbF^v\oplus\bbF$,
where $\set{\bfdelta_i}_{i=1}^{v}$ is the standard basis in $\bbF^v$,
\smash{$\set{\bfg_j}_{j=1}^{v+1}$} is a unimodular regular simplex for $\bbF^{v}$,
and $\set{c_j}_{j=1}^{v+1}$ is a unimodular sequence of scalars that is a Naimark complement for it~\cite{FickusJMP18}.

A direct computation reveals the inverse Welch bound~\eqref{equation.inverse Welch bound} of a Tremain ETF is always the integer $s=\frac12(v+3)=r+2$.
That said, a Tremain ETF need not contain any regular simplices.
Indeed, as detailed in~\cite{FickusJMP18},
when constructed with a \textit{Butson-type} Hadamard matrix of the appropriate parameters,
its inner products can be $p$th roots of unity where $p$ is odd,
at which point Corollary~\ref{corollary.empty triples} implies it has an empty binder.
Other Tremain ETFs do contain regular simplices.
For example, there is a $36$-vector Tremain ETF for $\bbR^{15}$ whose binder is a $\BIBD(36,6,8)$.
Because of this, it seems difficult to determine in general whether a Tremain ETF contains a regular simplex.
That said, since
\begin{equation*}
\tfrac{n}{s+1}
=\tfrac{(v+1)(v+2)}{v+5}=v-2+\tfrac{12}{v+5}
\end{equation*}
we can say in general that a Tremain ETF can only be a disjoint union of regular simplices in the special case where $v=7$.
Moreover, in general we have the following result:

\begin{theorem}
The Naimark complement of any Tremain ETF contains a regular simplex.
\end{theorem}

\begin{proof}
A Naimark complement of any Tremain ETF consists of $\frac12(v+1)(v+2)$ vectors in a space of dimension $\frac13v(v+2)$, meaning its inverse Welch bound is $v$.
We show that the last $v+1$ vectors in any Naimark complement of~\eqref{equation.definition of Tremain ETF} form a regular simplex.
Here, by multiplying by unimodular scalars if necessary, we assume without loss of generality that $c_j=1$ for all $j=1,\dotsc,v$,
meaning the Gram matrix of $\set{\bfg_j}_{j=1}^{v+1}$ is $\bfG^*\bfG=(v+1)\bfI-\bfone\bfone^*$.
As such, the Gram matrix of the last $v+1$ vectors in~\eqref{equation.definition of Tremain ETF} is
\begin{equation*}
\tfrac12\bfG^*\bfG+\tfrac32\bfone\bfone^*
=\tfrac12[(v+1)\bfI-\bfone\bfone^*]+\tfrac32\bfone\bfone^*
=\tfrac{v+1}2\bfone+\bfone\bfone^*.
\end{equation*}
Moreover, the tight frame constant of~\eqref{equation.definition of Tremain ETF} is
$a=\frac{n}{d}\norm{\bfphi_j}^2=\frac32(v+1)$.
Together, these facts imply that the Gram matrix of the last $v+1$ vectors in any Naimark complement of~\eqref{equation.definition of Tremain ETF} is
\begin{equation*}
a\bfI-(\tfrac12\bfG^*\bfG+\tfrac32\bfone\bfone^*)
=\tfrac{3(v+1)}2\bfI-(\tfrac{v+1}2\bfone+\bfone\bfone^*)
=(v+1)\bfI-\bfone\bfone^*,
\end{equation*}
meaning these vectors indeed form a regular $v$-simplex.
\end{proof}

Despite this fact, the Naimark complement of a Tremain ETF is never a disjoint union of regular simplices, since $v$ is necessarily odd, implying $v+1$ does not divide $n=\frac12(v+1)(v+2)$.
In general, we do not know when the binder of the Naimark complement of a Tremain ETF contains a BIBD.
In at least one case, it does:
computing the binder of the Naimark complement of the aforementioned $36$-vector Tremain ETF for $\bbR^{15}$ with BinderFinder, we find that it is a $\BIBD(36,8,6)$.
We leave a deeper investigation of these ETFs for future research.

\subsection{Phased BIBD ETFs}

As evidenced by Theorems~\ref{theorem.phased binder} and~\ref{theorem.Naimark complement of phased BIBD ETF}, we can say a lot about the regular simplices contained in the Naimark complement of a phased BIBD ETF.
In particular, Theorem~\ref{theorem.Naimark complement of phased BIBD ETF} gives that each row of a phased BIBD ETF indicates a regular simplex in its Naimark complements.
As such, the binder of the Naimark complement of any phased BIBD ETF indeed contains that BIBD.

Sometimes, these BIBDs contain a \textit{parallel class}, that is, a subcollection of blocks that form a partition for $[n]$.
In such cases, its Naimark complements are disjoint unions of regular simplices.
Not every binder has this property: for example, the binder~\eqref{equation.5x10 binder}  of~\eqref{equation.5x10 ETF} is a $\BIBD(10,4,2)$, but does not contain a parallel class since $4$ does not divide $10$.
That said,
as explained in~\cite{FickusJMPW19},
the phased $\BIBD(q^2,q,1)$ ETFs and phased $\BIBD(q^3+1,q+1,1)$ ETFs constructed there do always contain a parallel class.
In particular, whenever $q$ is an odd prime power,
there is a real phased $\BIBD(q^3+1,q+1,1)$ ETF whose underlying BIBD contains a parallel class,
meaning any one of its Naimark complements is a disjoint union of \smash{$\frac{q^3+1}{q+1}=q^2-q+1$} regular $q$-simplices.
Applying Theorem~\ref{theorem.equichordal} to this ETF produces an ECTFF for $\bbR^{q^2-q+1}$ consisting of $q^2-q+1$ subspaces of dimension $q$.

Interestingly, the existence of such a real ECTFF also follows from applying the techniques of~\cite{Zauner99} to a projective plane of order $q-1$, namely a $\BIBD(q^2-q+1,q,1)$.
However, such projective planes are only known to exist when $q-1$ is a prime power,
whereas our approach here assumes $q$ is an odd prime power.
In particular, when $q=7$, the Bruck-Ryser-Chowla theorem implies there does not exist a $\BIBD(43,7,1)$, and nevertheless this approach implies the existence of an ECTFF for $\bbR^{43}$ that consists of $43$ subspaces of dimension $7$.

One may also ask whether a phased BIBD ETF itself contains a regular simplex.
It turns out that this is not necessarily the case.
For example, the inner products of the vectors in the phased $\BIBD(q^3+1,q+1,1)$ ETF of~\cite{FickusJMPW19} are $(q+1)$th roots of unity.
When $q$ is an even prime power, Corollary~\ref{corollary.empty triples} then implies this ETF contains no regular simplices.
Other phased BIBD ETFs do contain regular simplices.
For example, the binder of the phased BIBD ETF~\eqref{equation.5x10 phased binder} is itself a $\BIBD(10,4,2)$.
For another example, the binder of~\eqref{equation.6x16 harmonic ETF} is a $\BIBD(16,4,3)$ and the binder of any one of its Naimark complements is a $\BIBD(16,6,2)$.
As such, there is a phased $\BIBD(16,4,3)$ ETF which contains $16$ regular $6$-simplices,
and it has a Naimark complement which is a $\BIBD(16,6,2)$ ETF which contains $60$ regular $4$-simplices.
One thing we note in this example is that the ``$v$" and ``$k$" parameters of these BIBDs are related according to $16-1=(4-1)(6-1)$.
This is not a coincidence:
by Theorem~\ref{theorem.Naimark complement of phased BIBD ETF},
a phased $\BIBD(n,k,\lambda)$ ETF spans a space of dimension~\eqref{equation.dimension of phased BIBD ETF},
and the corresponding inverse Welch bound~\eqref{equation.inverse Welch bound} can be found to be $s=\frac{n-1}{k-1}$.
Apart from this minor observation, we leave the study of the binders of phased BIBD ETFs for future work.

\subsection{SIC-POVMs}

\textit{Gerzon's bound} states that if \smash{$\set{\bfphi_j}_{j=1}^{n}$} is any sequence of noncollinear equiangular vectors in $\bbF^d$,
then $n$ is at most the dimension of the real Hilbert space of all $d\times d$ self-adjoint matrices.
That is, $n\leq\binom{d+1}{2}$ when $\bbF=\bbR$, while $n\leq d^2$ when $\bbF=\bbC$.
This follows from the fact that their outer products $\set{\bfphi_j^{}\bfphi_j^*}_{j=1}^{n}$ lie in this space and are linearly independent: we have
\begin{equation*}
\ip{\bfphi_j^{}\bfphi_j^*}{\bfphi_{j'}^{}\bfphi_{j'}^*}_{\Fro}
=\Tr(\bfphi_j^{}\bfphi_j^*\bfphi_{j'}^{}\bfphi_{j'}^*)
=\Tr(\bfphi_j^*\bfphi_{j'}^{}\bfphi_{j'}^*\bfphi_j^{})
=\abs{\ip{\bfphi_j}{\bfphi_{j'}}}^2
\end{equation*}
for any $j,j'=1,\dotsc,v$, meaning the Gram matrix of $\set{\bfphi_j^{}\bfphi_j^*}_{j=1}^{n}$ is an invertible matrix of the form $(c-w)\bfI+w\bfone\bfone^*$ for some $w<c$.
By definition, a SIC-POVM is an ETF for $\bbC^d$ that achieves this bound, namely an ETF with $n=d^2$ vectors.
Such ETFs arise in quantum information theory~\cite{Zauner99,RenesBSC04},
and Zauner has famously conjectured that they exist for any $d$~\cite{Zauner99}.
They are now known to exist when $d\in\set{2,\dotsc,24,28,30,31,35,37,39,43,48,124}$,
and there is strong numerical evidence for their existence in many other cases;
see~\cite{FuchsHS17} for a recent survey.

SIC-POVMs seldom contain regular simplices.
Indeed, the inverse Welch bound~\eqref{equation.inverse Welch bound} in this case is \smash{$(d+1)^{\frac12}$},
and so this can only happen when $d+1$ is a perfect square.
Our hope is that the techniques of this paper might help inform the search for SIC-POVMs in this special family.
In particular, it is known that when $d=3$, there exists a $9$-vector Steiner ETF for $\bbC^3$ arising from a $\BIBD(3,2,1)$.
Moreover, a SIC-POVM with $d=3$ arises as the Naimark complement of the phased $\BIBD(9,3,1)$ ETF~\eqref{equation.6x9 phased BIBD ETF}, and each of the $12$ blocks of this BIBD indicate a regular simplex in this ETF.
We have also applied BinderFinder to the SIC-POVM with $d=8$ from~\cite{Hoggar98} which arises by taking all translates and modulates of a single vector indexed by $\bbZ_2^3$.
Doing so reveals its binder to be a $\BIBD(64,4,3)$, meaning this $64$-vector ETF for $\bbC^8$ contains $1008$ regular simplices.
See~\cite{DangBBA13,Malikiosis16,ApplebyBDF17} for other work on linear dependent subcollections of SIC-POVMs.

The inverse Welch bound for the Naimark complement of a SIC-POVM for $\bbC^d$ is $(d-1)(d+1)^{\frac12}$.
Applying BinderFinder to the Naimark complements of the aforementioned SIC-POVMs with $d=3$ and $d=8$ reveals that they contain no regular simplices.

\section*{Acknowledgments}
This work was partially supported by NSF DMS 1321779, AFOSR F4FGA06060J007, AFOSR
Young Investigator Research Program award F4FGA06088J001, and the U.S.\ Air Force Research Lab Summer Faculty Fellowship Program.
Add to acknowledgements: Emily J.~King was supported in part by a Zentrum f\"ur Forschungsf\"orderung der Uni Bremen Explorationsprojekt ``Hilbert Space Frames and Algebraic Geometry.''
This project began during the Workshop on Frames and Algebraic \& Combinatorial Geometry in Bremen, Germany in July 2015,
and continued as part of the Summer of Frame Theory I in Dayton, Ohio in the summer of 2016.
The views expressed in this article are those of the authors and do not reflect the official policy or position of the United States Air Force, Department of Defense, or the U.S.~Government.

\end{document}